\documentclass[12pt]{amsart}
\usepackage{amsmath, amssymb, amsfonts}
\usepackage[all]{xypic}
\usepackage[pdftex]{graphicx}

\theoremstyle{plain}
\newtheorem{theorem}{Theorem}[section]
\newtheorem{lemma}[theorem]{Lemma}
\newtheorem{corollary}[theorem]{Corollary}
\newtheorem{prop}[theorem]{Proposition}

\newtheorem{conj}[theorem]{Conjecture}

\theoremstyle{remark}

\newtheorem{remark}[theorem]{Remark}

\newtheorem{example}[theorem]{Example}

\newtheorem*{note*}{Note}
\newtheorem*{remark*}{Remark}
\newtheorem*{example*}{Example}

\theoremstyle{definition}
\newtheorem*{definition*}{Definition}
\newtheorem*{hypothesis*}{Hypothesis}
\newtheorem*{assumptions*}{Assumptions}
\newtheorem{definition}[theorem]{Definition}

\newcommand{\Z}{\mathbb{Z}}

\newcommand{\Q}{\mathbb{Q}}
\newcommand{\C}{\mathbb{C}}
\newcommand{\N}{\mathbb{N}}
\newcommand{\F}{\mathbb{F}}
\newcommand{\Aff}{\mathrm{Aff}}

\newcommand{\Aut}{\mathrm{Aut}}
\newcommand{\Gal}{\mathrm{Gal}}

\newcommand{\GL}{\mathrm{GL}}

\newcommand{\nr}{\mathrm{nr}}
\newcommand{\End}{\mathrm{End}}
\newcommand{\Hom}{\mathrm{Hom}}

\newcommand{\res}{\mathrm{res}}
\newcommand{\quot}{\mathrm{quot}}

\newcommand{\Irr}{\mathrm{Irr}}
\newcommand{\ind}{\mathrm{ind}}
\newcommand{\PMod}{\mathrm{PMod}}

\newcommand{\Spec}{\mathrm{Spec}}
\newcommand{\ram}{\mathrm{ram}}
\newcommand{\ab}{\mathrm{ab}}
\newcommand{\Br}{\mathrm{Br}}

\numberwithin{equation}{section}

\newcommand{\mal}{^{\times}}
\newcommand{\et}{\mathrm{\acute{e}t}}

\newcommand{\perf}{\mathrm{perf}}
\newcommand{\tor}{_{\mathrm{tor}}}
\newcommand{\Det}{\mathrm{Det}}
\newcommand{\aug}{\mathrm{aug}}
\newcommand{\infl}{\mathrm{infl}}

\newcommand{\Dic}{\mathrm{Dic}}

\title[Hybrid Iwasawa algebras and the equivariant Iwasawa main conjecture]{Hybrid Iwasawa algebras and the \\
equivariant Iwasawa main conjecture}

\author{Henri Johnston}
\address{
Department of Mathematics\\
University of Exeter\\
Exeter\\
EX4 4QF\\
U.K.
}
\email{H.Johnston@exeter.ac.uk}
\urladdr{http://emps.exeter.ac.uk/mathematics/staff/hj241}

\author{Andreas Nickel}
\address{
Universit\"{a}t Bielefeld\\
Fakult\"{a}t f\"{u}r Mathematik\\
Postfach 100131\\
Universit\"{a}tsstr. 25\\
33501 Bielefeld\\
Germany}
\email{anickel3@math.uni-bielefeld.de}
\urladdr{http://www.math.uni-bielefeld.de/$\sim$anickel3/english.html}

\date{Version of 25th May 2016}
\begin{document}

\maketitle

\begin{abstract}
Let $p$ be an odd prime.
We give an unconditional proof of the equivariant Iwasawa main conjecture for totally real fields
for an infinite class of one-dimensional non-abelian $p$-adic Lie extensions.
Crucially, this result does not depend on the vanishing of the relevant Iwasawa $\mu$-invariant.
\end{abstract}


\section{Introduction}

Let $p$ be an odd prime.
Let $K$ be a totally real number field and let $K_{\infty}$ be the cyclotomic $\Z_{p}$-extension of $K$.
An admissible $p$-adic Lie extension $\mathcal{L}$ of $K$ is a Galois extension $\mathcal{L}$ of $K$ such that 
(i) $\mathcal{L}/K$ is unramified outside a finite set of primes $S$ of $K$,
(ii) $\mathcal{L}$ is totally real, (iii) $\mathcal{G} := \Gal(\mathcal{L}/K)$ is a compact 
$p$-adic Lie group, and (iv) $\mathcal{L}$ contains $K_{\infty}$. 
Let $M_{S}^{\ab}(p)$ be the maximal abelian $p$-extension of $\mathcal{L}$ unramified outside the set of primes above $S$.
Let $\Lambda(\mathcal{G}):=\Z_{p}[[\mathcal{G}]]$ denote the Iwasawa algebra of $\mathcal{G}$ over $\Z_{p}$
and let $X_{S}$ denote the (left) $\Lambda(\mathcal{G})$-module $\Gal(M_{S}^{\ab}(p) / \mathcal{L})$. 
Roughly speaking, the equivariant Iwasawa main conjecture (EIMC) relates $X_{S}$ to special values of Artin $L$-functions
via $p$-adic $L$-functions. This relationship can be expressed as the existence of a certain element in an algebraic $K$-group;
it is also conjectured that this element is unique.
 
There are at least three different versions of the EIMC. 
The first is due to Ritter and Weiss and deals with the case where $\mathcal{G}$ is a one-dimensional $p$-adic Lie group \cite{MR2114937},
and was proven  under a certain `$\mu=0$' hypothesis in a series of articles culminating in \cite{MR2813337}.
The second version follows the framework of Coates, Fukaya, Kato, Sujatha and Venjakob \cite{MR2217048} and 
was proven by Kakde \cite{MR3091976}, again assuming $\mu=0$.
This version is for $\mathcal{G}$ of arbitrary (finite) dimension and Kakde's proof uses a strategy of Burns and Kato to reduce to the one-dimensional case
(see Burns \cite{MR3294653}). Finally, Greither and Popescu \cite{MR3383600} have formulated and proven an EIMC via the Tate module of a certain Iwasawa-theoretic abstract $1$-motive, but they restricted their formulation to one-dimensional abelian extensions and the formulation itself
requires a $\mu=0$ hypothesis.
In \cite{MR3072281}, the second named author generalised this formulation (again assuming $\mu=0$) to the one-dimensional non-abelian case, and in the situation that all three formulations make sense (i.e.\ $\mathcal{G}$ is a one-dimensional $p$-adic Lie group and $\mu=0$), 
he showed that they are in fact all equivalent.
Venjakob \cite{MR3068897} has also compared the work of Ritter and Weiss to that of Kakde.

The classical Iwasawa $\mu=0$ conjecture (at $p$) is the assertion that for every number field $F$, the
Galois group of the maximal unramified abelian $p$-extension of $F_{\infty}$ is finitely generated as a $\Z_{p}$-module.
This conjecture was proven by Ferrero and Washington \cite{MR528968} in the case that $F/\Q$ is abelian, but little progress has been made since.
The $\mu=0$ hypothesis for an admissible $p$-adic Lie extension $\mathcal{L}/K$ discussed in the paragraph above 
is implied by the classical Iwasawa $\mu=0$ conjecture (at $p$) for a certain finite extension of $K$.
It follows from the result of Ferrero and Washington that the EIMC holds unconditionally when $\mathcal{L}/K$ 
is an admissible pro-$p$ extension and $K$ is abelian over $\Q$.

We wish to prove the EIMC in cases in which the $\mu=0$ hypothesis is not known. 
As a consequence, we must restrict to the case in which $\mathcal{G}$ is one-dimensional because the
`$\mathfrak{M}_{H}(G)$-conjecture' is required to even formulate the EIMC. 
This conjecture is known unconditionally in the one-dimensional case but in the more general case it is 
presently only known to hold under the $\mu=0$ hypothesis (see \S \ref{subsec:higher-rk}).
We remark that in the one-dimensional case, the $\mu=0$ hypothesis is equivalent to 
the Iwasawa module $X_{S}$ being finitely generated as a $\Z_{p}$-module; 
a detailed discussion of the relation to the classical Iwasawa $\mu=0$ conjecture in this setting is given in Remark \ref{rmk:mu=0}.

The main result of this article is the unconditional proof of the EIMC for an infinite class of admissible extensions with Galois group $\mathcal{G}$
a one-dimensional non-abelian $p$-adic Lie group and for which the $\mu=0$ hypothesis is not known.
A key ingredient is a result of Ritter and Weiss \cite{MR2114937} which, roughly speaking,
says that a version of the EIMC `over maximal orders' (or `character by character') holds without any $\mu=0$ hypothesis.
The proof uses Brauer induction to reduce to the abelian case, which is essentially equivalent to
the Iwasawa main conjecture for totally real fields proven by Wiles \cite{MR1053488}. 
Another key ingredient is the notion of `hybrid Iwasawa algebras' which is an adaptation of 
the notion of `hybrid $p$-adic group rings' first introduced by the present authors in \cite{MR3461042}.
Let $p$ be a prime and $G$ be a finite group with normal subgroup $N$.
The group ring $\Z_{p}[G]$ is said to be `$N$-hybrid' if $\Z_{p}[G]$ is isomorphic to the direct product of $\Z_{p}[G/N]$ and a maximal order.
Now suppose that $p$ is odd and that $\mathcal{G}$ is a one-dimensional $p$-adic Lie group with a finite normal subgroup $N$.
Then the Iwasawa algebra $\Lambda(\mathcal{G})$ is said to be `$N$-hybrid' 
if it decomposes into a direct product of $\Lambda(\mathcal{G}/N)$ and a maximal order.
By using the maximal order variant of the EIMC and certain functoriality properties, we show that the EIMC for the 
full extension (corresponding to $\mathcal{G}$) is equivalent to the EIMC for the sub-extension corresponding to $\mathcal{G}/N$.
There are many cases in which $\mu=0$ is not known for the full extension, but is known for the sub-extension,
and thus we obtain new unconditional results. 
However, we first need explicit criteria for $\Lambda(\mathcal{G})$ to be $N$-hybrid.
Since $\mathcal{G}$ is one-dimensional it decomposes as a semidirect product $\mathcal{G} = H \rtimes \Gamma$
where $H$ is finite and $\Gamma$ is isomorphic to $\Z_{p}$. 
Moreover, if $N$ is a finite normal subgroup of $\mathcal{G}$ then it must in fact be a (normal) subgroup of $H$.
We show that $\Lambda(\mathcal{G})$ is $N$-hybrid if and only if $\Z_{p}[H]$ is $N$-hybrid;
this is easy to see in the case that $\mathcal{G}$ is a direct product $H \times \Gamma$, 
but much more difficult in the general case. 
In \cite{MR3461042} we gave explicit criteria for $\Z_{p}[H]$ to be $N$-hybrid in terms of the degrees of the complex irreducible characters of $H$,
and thus the same criteria can be used to determine whether $\Lambda(\mathcal{G})$ is $N$-hybrid. 
We also study the behavior of $N$-hybrid
$p$-adic group rings $\Z_{p}[G]$ when we change the group $G$ and its normal subgroup $N$.
(As discussed in Remark \ref{rmk:applications-to-ETNC}, these new results on hybrid $p$-adic group rings also have applications 
to the equivariant Tamagawa number conjecture.)
\\ 

In \cite{non-abelian-Brumer-Stark}, we show that the EIMC implies the $p$-primary parts of refinements of the
imprimitive Brumer and Brumer-Stark conjectures.
These latter conjectures are classical when the relevant Galois group is abelian, and have been generalised to 
the non-abelian case by the second named author \cite{MR2976321} (independently,
they have been formulated in even greater generality by Burns \cite{MR2845620}).
By combining this with the main result of the present article, we give unconditional proofs of 
the non-abelian Brumer and Brumer-Stark conjectures in many cases. 
\\

This article is organised as follows. In \S \ref{sec:hybrid-group-rings} we review some material on hybrid
$p$-adic group rings, and show how new examples of such group rings can be obtained from existing examples.
In \S \ref{sec:hybrid-Iwasawa-alg} we generalise this notion to Iwasawa algebras and study its structure and basic properties.
In particular, we give explicit criteria for an Iwasawa algebra to be $N$-hybrid.
This enables us to give many examples of one-dimensional non-abelian $p$-adic Lie groups $\mathcal{G}$
such that the Iwasawa algebra $\Lambda(\mathcal{G})$ is $N$-hybrid for a non-trivial finite normal subgroup
$N$ of $\mathcal{G}$. 
In \S \ref{sec:EIMC} we give a slight reformulation of the EIMC that is convenient for our purposes. 
We also recall the functorial properties of the EIMC and reinterpret the maximal order variant of the EIMC of Ritter and Weiss \cite{MR2114937}.
The algebraic preparations of \S \ref{sec:hybrid-Iwasawa-alg} then permit us to verify the EIMC for many one-dimensional non-abelian
$p$-adic Lie extensions without assuming the $\mu=0$ hypothesis.

\subsection*{Acknowledgements}
It is a pleasure to thank Werner Bley, Ted Chinburg, Takako Fukaya, Lennart Gehrmann, Cornelius Greither, Annette Huber-Klawitter, 
Mahesh Kakde, Kazuya Kato, Daniel Macias Castillo, Cristian Popescu, J\"urgen Ritter, Sujatha, Otmar Venjakob, Christopher Voll,
Al Weiss and Malte Witte for helpful discussions and correspondence.
The authors also thank the referee for several helpful comments.
The second named author acknowledges financial support provided by the DFG within the Collaborative Research Center 701
`Spectral Structures and Topological Methods in Mathematics'.

\subsection*{Notation and conventions}
All rings are assumed to have an identity element and all modules are assumed
to be left modules unless otherwise  stated. We fix the following notation:

\medskip

\begin{tabular}{ll}
$S_{n}$ & the symmetric group on $n$ letters\\
$A_{n}$ & the alternating group on $n$ letters\\
$C_{n}$ & the cyclic group of order $n$\\
$D_{2n}$ & the dihedral group of order $2n$\\
$Q_{8}$ & the quaternion group of order $8$\\
$V_{4}$ & the subgroup of $A_{4}$ generated by double transpositions\\
$\F_{q}$ & the finite field with $q$ elements, where $q$ is a prime power\\
$\Aff(q)$ & the affine group isomorphic to $\F_{q} \rtimes \F_{q}^{\times}$ defined in Example \ref{ex:affine}\\
$v_{p}(x)$ & the $p$-adic valuation of $x \in \Q$\\
$R^{\times}$ & the group of units of a ring $R$\\
$\zeta(R)$ & the centre of a ring $R$\\
$M_{m \times n} (R)$ & the set of all $m \times n$ matrices with entries in a ring $R$\\
$\zeta_{n}$ & a primitive $n$th root of unity\\
$K_{\infty}$ & the cyclotomic $\Z_{p}$-extension of the number field $K$\\
$K^{+}$ & the maximal totally real subfield of $K$\\
$K^{c}$ & an algebraic closure of a field $K$ \\
$\Irr_{F}(G)$ & the set of $F$-irreducible characters of the (pro)-finite group $G$\\
& (with open kernel) where $F$ is a field of characteristic $0$

\end{tabular}

\section{Hybrid $p$-adic group rings} \label{sec:hybrid-group-rings}

We recall material on hybrid $p$-adic group rings from \cite[\S 2]{MR3461042}
and prove new results which provide many new examples.
We shall sometimes abuse notation by using the symbol $\oplus$ to denote the direct product of rings or orders.

\subsection{Background material}
Let $p$ be a prime and let $G$ be a finite group.
For a normal subgroup $N \unlhd G$, let $e_{N} = |N|^{-1}\sum_{\sigma \in N} \sigma$
be the associated central trace idempotent in the group algebra $\Q_{p}[G]$.
Then there is a ring isomorphism $\Z_{p}[G]e_{N} \simeq \Z_{p}[G/N]$.
We now specialise \cite[Definition 2.5]{MR3461042} to the case of $p$-adic group rings
(we shall not need the more general case of $N$-hybrid orders).

\begin{definition}
Let $N \unlhd G$. We say that the $p$-adic group ring $\Z_{p}[G]$ is \emph{$N$-hybrid}
if (i) $e_{N} \in \Z_{p}[G]$ (i.e. $p \nmid |N|$) and (ii) $\Z_{p}[G](1-e_{N})$ is a maximal
$\Z_{p}$-order in $\Q_{p}[G](1-e_{N})$.
\end{definition}

\begin{remark}
The group ring $\Z_{p}[G]$ is itself maximal if and only if $p$ does not divide $|G|$
if and only if $\Z_{p}[G]$ is $G$-hybrid. Moreover, $\Z_{p}[G]$ is always $\{1\}$-hybrid.
\end{remark}

For every field $F$ of characteristic $0$ and every finite group $G$, we denote by $\Irr_{F}(G)$
the set of $F$-irreducible characters of $G$.
Let $\Q_{p}^{c}$ be an algebraic closure of $\Q_{p}$.
If $x$ is a rational number, we let $v_{p}(x)$ denote its $p$-adic valuation.

\begin{prop}[{\cite[Proposition 2.7]{MR3461042}}]\label{prop:hybrid-criterion-groupring}
The group ring $\Z_{p}[G]$ is $N$-hybrid if and only if
for every $\chi \in \Irr_{\Q_{p}^{c}}(G)$ such that $N \not \leq \ker \chi$
we have $v_{p}(\chi(1))=v_{p}(|G|)$.
\end{prop}

\begin{remark}
In the language of modular representation theory, when $v_{p}(\chi(1))=v_{p}(|G|)$
we say that ``$\chi$ belongs to a $p$-block of defect zero''.
\end{remark}

\subsection{Frobenius groups}\label{subsec:frobenius-groups}
We recall the definition and some basic facts about Frobenius groups and then use them to
provide many examples of hybrid group rings.
For further results and examples, we refer the reader to \cite[\S 2.3]{MR3461042}.

\begin{definition}
A \emph{Frobenius group} is a finite group $G$ with a proper non-trivial subgroup $H$
such that $H \cap gHg^{-1}=\{ 1 \}$ for all $g \in G-H$,
in which case $H$ is called a \emph{Frobenius complement}.
\end{definition}

\begin{theorem}\label{thm:frob-kernel}
A Frobenius group $G$ contains a unique normal subgroup $N$, known as the Frobenius kernel, such that
$G$ is a semidirect product $N \rtimes H$. Moreover:
\begin{enumerate}
\item $|N|$ and $[G:N]=|H|$ are relatively prime.
\item The Frobenius kernel $N$ is nilpotent.
\item If $K \unlhd G $ then either $K \unlhd N$ or $N \unlhd K$.
\item If $\chi \in \Irr_{\C}(G)$ such that  $N \not \leq \ker \chi$ then $\chi= \ind_{N}^{G}(\psi)$ for some $1 \neq \psi \in \Irr_{\C}(N)$.
\end{enumerate}
\end{theorem}

\begin{proof}
For (i) and (iv) see \cite[\S 14A]{MR632548}.
For (ii) see \cite[10.5.6]{MR1357169} and for (iii) see  \cite[Exercise 7, \S 8.5]{MR1357169}.
\end{proof}

\begin{prop}[{\cite[Proposition 2.13]{MR3461042}}]\label{prop:frob-N-hybrid}
Let $G$ be a Frobenius group with Frobenius kernel $N$.
Then for every prime $p$ not dividing $|N|$, the group ring $\Z_{p}[G]$ is $N$-hybrid.
\end{prop}

\begin{proof}
We repeat the short argument for the convenience of the reader.
Let $\chi \in \Irr_{\Q_{p}^{c}}(G)$ such that $N \not \leq \ker \chi$.
Then by Theorem \ref{thm:frob-kernel} (iv) $\chi$ is induced from a nontrivial irreducible character
of $N$ and so $\chi(1)$ is divisible by $[G:N]$.
However, $|N|$ and $[G:N]$ are relatively prime by Theorem \ref{thm:frob-kernel} (i) and so
Proposition \ref{prop:hybrid-criterion-groupring} now gives the desired result.
\end{proof}

We now give some examples, the first two of which were also given in \cite[\S 2.3]{MR3461042}.

\begin{example}\label{ex:metacyclic}
Let $p<q$ be distinct primes and assume that $p \mid (q-1)$.
Then there is an embedding $C_{p} \hookrightarrow \Aut(C_{q})$ and so there is
a fixed-point-free action of $C_{p}$ on $C_{q}$.
Hence the corresponding semidirect product $G = C_{q} \rtimes C_{p}$ is a Frobenius group
(see \cite[Theorem 2.12]{MR3461042} or \cite[\S 4.6]{MR2599132}, for example), and so $\Z_{p}[G]$ is $N$-hybrid with $N = C_{q}$.
\end{example}

\begin{example}\label{ex:affine}
Let $q$ be a prime power and let $\F_{q}$ be the finite field with $q$ elements.
The group $\Aff(q)$ of affine transformations on $\F_{q}$ is the group of transformations
of the form $x \mapsto ax +b$ with $a \in \F_{q}^{\times}$ and $b \in \F_{q}$.
Let $G=\Aff(q)$ and let $N=\{ x \mapsto x+b \mid b \in \F_{q} \}$.
Then $G$ is a Frobenius group with Frobenius kernel $N \simeq \F_{q}$ and is isomorphic to
the semidirect product $\F_{q} \rtimes \F_{q}^{\times}$ with the natural action.
Moreover, $G/N \simeq \F_{q}^{\times} \simeq C_{q-1}$ and $G$ has precisely one
non-linear irreducible complex character, which is rational-valued and of degree $q-1$.
Hence for every prime $p$ not dividing $q$, we have that $\Z_{p}[G]$ is $N$-hybrid
and is isomorphic to $\Z_{p}[C_{q-1}] \oplus M_{(q-1) \times (q-1)}(\Z_{p})$.
Note that in particular $\Aff(3) \simeq S_{3}$ and $\Aff(4) \simeq A_{4}$.
Thus $\Z_{2}[S_{3}] \simeq \Z_{2}[C_{2}] \oplus M_{2 \times 2}(\Z_{2})$
and $\Z_{3}[A_{4}] \simeq \Z_{3}[C_{3}] \oplus M_{3 \times 3}(\Z_{3})$.
\end{example}

\begin{example}\label{ex:dicyclic-complement}
Let $p$ be an odd prime and let $\Dic_{p} := \langle a,b \mid a^{2p}=1, a^{p}=b^{2}, b^{-1}ab = a^{-1} \rangle$ be the dicyclic group of order $4p$.
We recall a construction given in \cite[Chapter 14]{MR1828640}.
For every positive integer $n$ we let $\zeta_{n}$ denote a primitive $n$th root of unity.
There is an embedding $\iota$ of $\Dic_{p}$ into the subring $S_{p} := \Z[\frac{1}{2p}, \zeta_{2p},j]$ of the real quaternions;
here, $j^{2} = -1$ and $j \zeta_{2p} = \zeta_{2p}^{-1} j$.
We put $R_{p} := \Z[\frac{1}{2p}, \zeta_{p} + \zeta_{p}^{-1}] \subseteq S_{p}$.
Let $t$ and $k_{i}$, $1\leq i \leq t$ be positive integers and let $\mathfrak{p}_{i}$, $1 \leq i \leq t$ be maximal ideals of $R_{p}$.
Then for each $i$ there is a $R_{p} / \mathfrak{p}_{i}^{k_{i}}$-algebra isomorphism
$S_{p} / \mathfrak{p}_{i}^{k_{i}} S_{p} \simeq M_{2 \times 2} (R_{p}/\mathfrak{p}_{i}^{k_{i}})$ which
induces a fixed-point-free action of $\Dic_{p}$ on $N(\mathfrak{p}_{i}^{k_{i}}) := (R_{p} / \mathfrak{p}_{i}^{k_{i}})^{2}$ via $\iota$.
Thus $G := N \rtimes \Dic_{p}$ with $N := \prod_{i=1}^{t} N(\mathfrak{p}_{i}^{k_{i}})$ is a Frobenius group with Frobenius complement $\Dic_{p}$.
In fact, every Frobenius group with Frobenius complement isomorphic to $\Dic_{p}$
is of this type (see \cite[Theorem 14.4]{MR1828640}).
In particular, $\Z_{p}[G]$ is $N$-hybrid.
\end{example}

\begin{example} \label{ex:non-abelian-kernel}
Let $p$ and $q$ be primes, $f$ and $n$ positive integers such that $q>n>1$ and $q \mid (p^{f}-1)$.
Let $N$ be the subgroup of $\GL_{n}(\F_{p^{f}})$ comprising upper triangular matrices with all diagonal
entries equal to $1$.
There are pairwise distinct $b_{j} \in \F_{p^{f}}$, $1 \leq j \leq n$ such that $b_{j}^{q} = 1$.
Let $h$ be the diagonal matrix with entries $b_{1}, \dots, b_{n}$ and set $H:= \langle h \rangle$.
Then $G := N \rtimes H$ is a Frobenius group  of order $q p^{fn(n-1)/2}$ and so $\Z_{p}[G]$ is $N$-hybrid.
Moreover, the Frobenius kernel $N$ has nilpotency class $n-1$ (see \cite[Example 16.8b]{MR1645304})
and hence is complicated if $n$ is large.
\end{example}

\subsection{New $p$-adic hybrid group rings from old}
We now use character theory to show how new examples of hybrid $p$-adic group rings can be obtained from existing examples.

For a field $F$ of characteristic $0$, a finite group $G$ and (virtual) $F$-characters $\chi$ and $\psi$ of $G$,
we let $\langle \chi, \psi \rangle$ denote the usual inner product.

\begin{remark}
For any finite group $G$ with subgroup $H$, any prime $p$,
and any $\chi \in \Irr_{\Q_{p}^{c}}(G)$, we have $H \leq \ker\chi$ if and only if $\langle \res^{G}_{H} \chi, \psi \rangle = 0$
for every non-trivial $\psi \in \Irr_{\Q_{p}^{c}}(H)$. We shall use this easy observation several times (for different choices of
$G$, $H$ and $\chi$) in the proofs of Lemma \ref{lem:kernel-basechange} and Propositions \ref{prop:hybrid-basechange-down} and
\ref{prop:hybrid-basechange-up} below. 
\end{remark}

\begin{lemma}\label{lem:kernel-basechange}
Let $G$ be a finite group with subgroups $N \leq H \unlhd  G$.
Let $p$ be a prime.
Fix $\chi \in \Irr_{\Q_{p}^{c}}(G)$ and let $\eta \in \Irr_{\Q_{p}^{c}}(H)$ be an irreducible constituent of $\res^{G}_{H} \chi$.
Then:
\begin{enumerate}
\item If $N \leq \ker \chi$, then $N \leq \ker \eta$.
\item Assume in addition that $N \unlhd G$. Then $N \leq \ker \chi$ if and only if $N \leq \ker \eta$.
\end{enumerate}
\end{lemma}

\begin{proof}
As $H \unlhd G$, we have a natural right action of $G$ on $\Irr_{\Q_{p}^{c}}(H)$;
namely, for each $g \in G$, $h \in H$ we have $\eta^{g}(h) = \eta(g^{-1}hg)$.
Let $G_{\eta} = \{g \in G \mid \eta^{g} = \eta \}$ be the stabiliser of $\eta$ in $G$ and
let $R_{\eta}$ be a set of right coset representatives of $G_{\eta}$ in $G$.
Then by Clifford theory (see \cite[Proposition 11.4]{MR632548}) we have
$\res^{G}_{H} \chi = e \sum_{g \in R_{\eta}} \eta^{g}$ for some positive integer $e$, and in particular
\begin{equation}\label{eqn:Clifford-theory}
\chi(1) = e [G:G_{\eta}] \eta(1).
\end{equation}
Let $\psi \in \Irr_{\Q_{p}^{c}}(N)$. Then we have
\begin{equation}\label{eqn:res-non-neg}
\langle \res^{G}_{N} \chi, \psi \rangle = \langle \res^{H}_{N} (\res^{G}_{H} \chi), \psi \rangle
= e \sum_{g \in R_{\eta}} \langle \res^{H}_{N} \eta^{g}, \psi \rangle \geq 0,
\end{equation}
since $\langle \res^{H}_{N} \eta^{g}, \psi \rangle \geq 0$ for every $g \in R_{\eta}$.

Suppose that $N \leq \ker \chi$.
Then $\langle \res^{G}_{N} \chi, \psi \rangle = 0$ for every non-trivial $\psi \in \Irr_{\Q_{p}^{c}}(N)$ and so by \eqref{eqn:res-non-neg} we have $\langle \res^{H}_{N} \eta, \psi \rangle = 0$. Hence $N \leq \ker \eta$, proving (i).

Now assume that $N \unlhd G$ and suppose conversely that $N \leq \ker \eta$.
Then $\langle \res^{H}_{N} \eta, \psi \rangle = 0$ for every non-trivial $\psi \in \Irr_{\Q_{p}^{c}}(N)$,
and for every $g \in R_{\eta}$ we have
\[
    \langle \res^{H}_{N} \eta^{g}, \psi \rangle = \langle \res^{H}_{N} \eta, \psi^{g^{-1}} \rangle = 0.
\]
Thus by \eqref{eqn:res-non-neg} we have $\langle \res^{G}_{N} \chi, \psi \rangle = 0$
for every non-trivial $\psi \in \Irr_{\Q_{p}^{c}}(N)$ and so  $N \leq \ker \chi$.
\end{proof}

The following proposition is a generalisation of \cite[Proposition 2.8 (iv)]{MR3461042}.

\begin{prop}\label{prop:hybrid-basechange-down}
Let $G$ be a finite group with normal subgroups $N, H \unlhd  G$.
Let $K$ be a normal subgroup of $H$ such that $K \leq N$. Let $p$ be a prime.
If  $\Z_{p}[G]$ is $N$-hybrid then $\Z_{p}[H]$ is $K$-hybrid.
\end{prop}

\begin{proof}
Fix $\eta \in \Irr_{\Q_{p}^{c}}(H)$ such that $K \not\leq \ker \eta$.
By Proposition \ref{prop:hybrid-criterion-groupring} we have to show that $v_{p}(\eta(1)) = v_{p}(|H|)$.
Let $\chi \in \Irr_{\Q_{p}^{c}}(G)$ be any irreducible constituent of $\ind^{G}_{H}  \eta$.
Then by Frobenius reciprocity we have  $\langle\chi, \ind^{G}_{H}  \eta \rangle = \langle \res^{G}_{H} \chi, \eta \rangle \neq 0$.
By Lemma \ref{lem:kernel-basechange} (i) we have $K \not\leq \ker \chi$ and a fortiori $N \not\leq \ker \chi$.
Then again by Proposition \ref{prop:hybrid-criterion-groupring} we find that
$v_{p}(\chi(1)) = v_{p}(|G|)$ as $\Z_{p}[G]$ is $N$-hybrid by assumption.
This and equation \eqref{eqn:Clifford-theory} imply $v_{p}(e \eta(1)) = v_{p}(|G_{\eta}|) = v_{p}([G_{\eta}:H] \cdot |H|)$.
But $\eta(1)$ divides $|H|$ whereas $e$ divides $[G_{\eta}:H]$ by \cite[Theorem 21.3]{MR1645304}, so we must have
$v_{p}(e) = v_{p}([G_{\eta}:H])$ and $v_{p}(\eta(1)) = v_{p}(|H|)$, as desired.
\end{proof}

The following proposition is a generalisation of \cite[Lemma 2.9]{MR3461042}.

\begin{prop} \label{prop:hybrid-basechange-up}
Let $G$ be a finite group with normal subgroups $N \unlhd H \unlhd G$ such that $N \unlhd G$.
Let $p$ be a prime and assume that $p \nmid [G:H]$.
Then $\Z_{p}[G]$ is $N$-hybrid if and only if $\Z_{p}[H]$ is $N$-hybrid.
\end{prop}

\begin{proof}
If $\Z_{p}[G]$ is $N$-hybrid then $\Z_{p}[H]$ is $N$-hybrid by Proposition \ref{prop:hybrid-basechange-down} with $K=N$.
Suppose conversely that $\Z_{p}[H]$ is $N$-hybrid and assume that $p \nmid [G:H]$.
Let $\chi \in \Irr_{\Q_{p}^{c}}(G)$ such that $N \not\leq \ker \chi$ and let $\eta$ be an irreducible constituent of $\res^{G}_{H} \chi$.
Then $N \not\leq \ker \eta$ by Lemma \ref{lem:kernel-basechange} (ii) and $\eta(1)$ divides $\chi(1)$ by \eqref{eqn:Clifford-theory}.
However, we have $v_{p}(\eta(1)) = v_{p}(|H|) = v_{p}(|G|)$ by assumption, and so we must also have that $v_{p}(\chi(1)) = v_{p}(|G|)$.
Thus $\Z_{p}[G]$ is $N$-hybrid by Proposition \ref{prop:hybrid-criterion-groupring}.
\end{proof}

\begin{example}\label{ex:S4-A4-V4}
Let $p=3$, $G=S_{4}$, $H=A_{4}$ and $N = V_{4}$. Then the hypotheses of Proposition \ref{prop:hybrid-basechange-up} are satisfied.
Hence $\Z_{3}[S_{4}]$ is $V_{4}$-hybrid if and only if $\Z_{3}[A_{4}]$ is $V_{4}$-hybrid.
In fact, $\Z_{3}[A_{4}]$ is indeed $V_{4}$-hybrid since $A_{4}$ is a Frobenius group with Frobenius kernel $V_{4}$
(see Example \ref{ex:affine}) and so $\Z_{3}[S_{4}]$ is also $V_{4}$-hybrid.
However, $S_{4}$ is \emph{not} a Frobenius group (see \cite[Example 2.18]{MR3461042}).
Thus Proposition \ref{prop:hybrid-basechange-up} can be used to give examples which do not come directly from
Proposition \ref{prop:frob-N-hybrid}.
\end{example}

\begin{example} \label{ex:affine-Frobenius}
Let $q = \ell^{n}$ be a prime power and let $\phi: \F_{q} \rightarrow \F_{q}$, $x \mapsto x^{\ell}$ be the Frobenius automorphism.
Each $c \in \F_{q}\mal$ defines a map $m_{c}: \F_{q} \rightarrow \F_{q}$, $x \mapsto c \cdot x$.
We may consider $\phi$ and $m_{c}$ as elements of $\GL_{n}(\F_{\ell})$. Then $\phi m_{c} \phi^{-1} = m_{c^{\ell}}$
and we may form the semidirect product $\F_{q}\mal \rtimes \langle \phi \rangle$ inside $\GL_{n}(\F_{\ell})$.
Moreover, the action of $\F_{q}\mal \rtimes \langle \phi \rangle$ on $\F_{q}$ gives a semidirect product
$G := \F_{q} \rtimes (\F_{q}\mal \rtimes \langle \phi \rangle)$.
Then the group $\Aff(q) \simeq \F_{q} \rtimes \F_{q}^{\times}$ of affine transformations on $\F_{q}$
naturally identifies with a normal subgroup of $G$, and $N = \F_{q}$ is normal in both $G$ and $\Aff(q)$.
However, $\Z_{p}[\Aff(q)]$ is $N$-hybrid for every prime $p \neq  \ell$ by Example \ref{ex:affine}.
If we further suppose that $p$ does not divide $n = [G : \Aff(q)]$, then $\Z_{p}[G]$
is $N$-hybrid by Proposition \ref{prop:hybrid-basechange-up}.
Note that this recovers Example \ref{ex:S4-A4-V4} since $G \simeq S_{4}$ when $\ell=n=2$.
\end{example}

\begin{remark}\label{rmk:applications-to-ETNC}
Burns and Flach \cite{MR1884523} formulated
the equivariant Tamagawa number conjecture (ETNC)
for any motive over $\Q$ with the action of a semisimple $\Q$-algebra, 
describing the leading term at $s=0$ of an equivariant motivic $L$-function in terms of certain 
cohomological Euler characteristics. 
The present authors introduced hybrid $p$-adic groups rings in \cite{MR3461042} and used them to prove many new cases of the 
$p$-part of the ETNC for Tate motives; 
the same methods can also be applied to several related conjectures.
Thus the new results on $p$-adic group rings given here combined with the results of \cite{MR3461042}
give unconditional proofs of the $p$-part of the ETNC for Tate motives and related conjectures in many new cases.
\end{remark}

\section{Hybrid Iwasawa algebras}\label{sec:hybrid-Iwasawa-alg}

\subsection{Iwasawa algebras of one-dimensional $p$-adic Lie groups}\label{subsec:Iwasawa-algebras}
Let $p$ be an odd prime and let $\mathcal{G}$ be a profinite group containing a finite normal subgroup $H$ 
such that $\mathcal{G}/H \simeq \overline{\Gamma}$ where $\overline\Gamma$ is a pro-$p$-group isomorphic to $\Z_{p}$.
The argument given in \cite[\S 1]{MR2114937} shows that the short exact sequence
\[
1 \longrightarrow H \longrightarrow \mathcal{G} \longrightarrow \overline{\Gamma} \longrightarrow 1
\]
splits. Thus we obtain a semidirect product $\mathcal{G} = H \rtimes \Gamma$ where $\Gamma \leq \mathcal{G}$ and $\Gamma \simeq \overline{\Gamma} \simeq \Z_{p}$. In other words, $\mathcal{G}$ is  a one-dimensional $p$-adic Lie group.
The Iwasawa algebra of $\mathcal{G}$ is
\[
\Lambda(\mathcal{G}) := \Z_{p}[[\mathcal{G}]] = \varprojlim \Z_{p}[\mathcal{G}/\mathcal{N}],
\]
where the inverse limit is taken over all open normal subgroups $\mathcal{N}$ of $\mathcal{G}$.
If $F$ is a finite field extension of $\Q_{p}$  with ring of integers $\mathcal{O}=\mathcal{O}_{F}$,
we put $\Lambda^{\mathcal{O}}(\mathcal{G}) := \mathcal{O} \otimes_{\Z_{p}} \Lambda(\mathcal{G}) = \mathcal{O}[[\mathcal{G}]]$.
We fix a topological generator $\gamma$ of $\Gamma$.
Since any homomorphism $\Gamma \rightarrow \Aut(H)$ must have open kernel, we may choose a natural number $n$ such that $\gamma^{p^n}$ is central in $\mathcal{G}$; we fix such an $n$.
 As $\Gamma_{0} := \Gamma^{p^n} \simeq \Z_{p}$, there is a ring isomorphism
$R:=\mathcal{O}[[\Gamma_{0}]] \simeq \mathcal{O}[[T]]$ induced by $\gamma^{p^n} \mapsto 1+T$
where $\mathcal{O}[[T]]$ denotes the power series ring in one variable over $\mathcal{O}$.
If we view $\Lambda^{\mathcal{O}}(\mathcal{G})$ as an $R$-module (or indeed as a left $R[H]$-module), there is a decomposition
\begin{equation*}\label{eq:Lambda-R-decomp}
\Lambda^{\mathcal{O}}(\mathcal{G}) = \bigoplus_{i=0}^{p^n-1} R[H] \gamma^{i}.
\end{equation*}
Hence $\Lambda^{\mathcal{O}}(\mathcal{G})$ is finitely generated as an $R$-module and is an $R$-order in the separable $E:=Quot(R)$-algebra
$\mathcal{Q}^{F} (\mathcal{G})$, the total ring of fractions of $\Lambda^{\mathcal{O}}(\mathcal{G})$, obtained
from $\Lambda^{\mathcal{O}}(\mathcal{G})$ by adjoining inverses of all central regular elements.
Note that $\mathcal{Q}^{F} (\mathcal{G}) =  E \otimes_{R} \Lambda^{\mathcal{O}}(\mathcal{G})$ and that by
\cite[Lemma 1]{MR2114937} we have $\mathcal{Q}^{F} (\mathcal{G}) = F \otimes_{\Q_{p}} \mathcal{Q}(\mathcal{G})$,
where $\mathcal{Q}(\mathcal{G}) := \mathcal{Q}^{\Q_{p}}(\mathcal{G})$.

\subsection{Characters and central primitive idempotents} \label{subsec:idempotents}
For any field $K$ of characteristic $0$ let $\Irr_{K}(\mathcal{G})$ the set of $K$-irreducible characters of $\mathcal{G}$ with open kernel.
Fix a character $\chi \in \Irr_{\Q_{p}^{c}}(\mathcal{G})$ and let $\eta$ be an irreducible constituent of
$\res^{\mathcal{G}}_{H} \chi$.
Then $\mathcal{G}$ acts on $\eta$ as $\eta^{g}(h) = \eta(g^{-1}hg)$
for $g \in \mathcal{G}$, $h \in H$, and following \cite[\S 2]{MR2114937} we set
\[
St(\eta) := \{g \in \mathcal{G}: \eta^g = \eta \}, \quad e(\eta) := \frac{\eta(1)}{|H|} \sum_{h \in H} \eta(h^{-1}) h,
\quad e_{\chi} := \sum_{\eta \mid \res^{\mathcal{G}}_{H} \chi} e(\eta).
\]
By \cite[Corollary to Proposition 6]{MR2114937} $e_{\chi}$ is a primitive central idempotent of
$\mathcal{Q}^{c}(\mathcal{G}) := \Q_{p}^{c} \otimes_{\Q_{p}} \mathcal{Q}(\mathcal{G})$.
In fact, every primitive central idempotent of $\mathcal{Q}^{c}(\mathcal{G})$ is of this form
and $e_{\chi} = e_{\chi'}$ if and only if $\chi = \chi' \otimes \rho$ for some character $\rho$ of $\mathcal{G}$ of type $W$
(i.e.~$\res^{\mathcal{G}}_{H} \rho = 1$).
The irreducible constituents of $\res^{\mathcal{G}}_{H} \chi$ are precisely the conjugates of $\eta$
under the action of $\mathcal{G}$, each occurring with the same multiplicity $z_{\chi}$ by \cite[Proposition 11.4]{MR632548}. By \cite[Lemma 4]{MR2114937} we have $z_{\chi}=1$ and thus we also have equalities
\begin{equation}\label{eq:idem-sum}
\res^{\mathcal{G}}_{H} \chi = \sum_{i=0}^{w_{\chi}-1} \eta^{\gamma^{i}},
\quad
e_{\chi} = \sum_{i=0}^{w_{\chi}-1} e(\eta^{\gamma^{i}}) = \frac{\chi(1)}{|H|w_{\chi}}\sum_{h \in H} \chi(h^{-1})h,
\end{equation}
where $w_{\chi} := [\mathcal{G} : St(\eta)]$.
Note that  $\chi(1) = w_{\chi} \eta(1)$ and that
$w_{\chi}$ is a power of $p$ since $H$ is a subgroup of $St(\eta)$.

\subsection{Working over sufficiently large $p$-adic fields}\label{subsec:sufficiently-large}
We now specialise to the case where $F/\Q_{p}$ is a finite extension over which both
characters $\chi$ and $\eta$ have realisations.
Let $V_{\chi}$ denote a realisation of $\chi$ over $F$.	
By \cite[Proposition 5]{MR2114937}, there exists a unique element
$\gamma_{\chi} \in \zeta(\mathcal{Q}^{c}(\mathcal{G})e_{\chi})$ such that $\gamma_{\chi}$
acts trivially on $V_{\chi}$ and $\gamma_{\chi}= gc$ where $g \in \mathcal{G}$ with $(g \bmod H) = \gamma^{w_{\chi}}$
and $c \in (\Q_{p}^{c}[H]e_{\chi})^{\times}$. Moreover, $\gamma_{\chi}= gc=cg$.

\begin{lemma}\label{lem:c-in-F[H]}
In fact  $c \in (F[H]e_{\chi})^{\times}$ and so
$\gamma_{\chi} \in \zeta(\mathcal{Q}^{F}(\mathcal{G})e_{\chi})$.
\end{lemma}

\begin{proof}
We recall the definition of $\gamma_{\chi} := gc$ given in the proof of \cite[Proposition 5]{MR2114937}, the only difference
being that there $V_{\chi}$ is defined over $\Q_{p}^{c}$ rather than $F$.
Choose $g \in \mathcal{G}$ such that $(g \bmod H) = \gamma^{w_{\chi}}$.
Then $g \in St(\eta^{\gamma^{i}})$ for each $i$,
and it acts on $V_{\chi} = \oplus_{i=0}^{w_{\chi}-1} e(\eta^{\gamma^{i}}) V_{\chi}$ componentwise.
Since, by \cite[Lemma 4]{MR2114937}, each $e(\eta^{\gamma^{i}}) V_{\chi}$ is $H$-irreducible,
\[
g^{-1} \mid_{e(\eta^{\gamma^{i}})V_{\chi}} =: c(\eta^{\gamma^{i}}) \in F[H]e(\eta^{\gamma^{i}}) \simeq \End_{F}(e(\eta^{\gamma^{i}})V_{\chi})
\simeq M_{\eta(1) \times \eta(1)}(F).
\]
Now we set $c:=\sum_{i=0}^{w_{\chi}-1} c(\eta^{\gamma^{i}})$, and we see that $c \in (F[H]e_{\chi})^{\times}$.
\end{proof}

By \cite[Proposition 5]{MR2114937}, the element $\gamma_{\chi}$
generates a procyclic $p$-subgroup $\Gamma_{\chi}$ of $(\mathcal{Q}^{F}(\mathcal{G})e_{\chi})^{\times}$.
Let $\Lambda^{\mathcal{O}}(\Gamma_{\chi})$ be the integral domain $\mathcal{O}[[\Gamma_{\chi}]]$
with field of fractions $\mathcal{Q}^{F}(\Gamma_{\chi})$.

\begin{lemma}\label{lem:unique-max-order-in-centre}
$\Lambda^{\mathcal{O}}(\Gamma_{\chi})$ is the unique maximal $R$-order in $\zeta(\mathcal{Q}^{F}(\mathcal{G})e_{\chi})\simeq \mathcal{Q}^{F}(\Gamma_{\chi})$.
\end{lemma}

\begin{proof}
By \cite[Proposition 6]{MR2114937}, $\mathcal{Q}^{F}(\Gamma_{\chi})$ is contained in $\mathcal{Q}^{F}(\mathcal{G})e_{\chi}$,
and $\gamma_{\chi} \in \mathcal{Q}^{F}(\Gamma_{\chi})$ induces an isomorphism
$\mathcal{Q}^{F}(\Gamma_{\chi}) \stackrel{\simeq}{\longrightarrow} \zeta(\mathcal{Q}^{F}(\mathcal{G})e_{\chi})$.
Therefore $\Lambda^{\mathcal{O}}(\Gamma_{\chi})$ is an $R$-order in $\zeta(\mathcal{Q}^{F}(\mathcal{G})e_{\chi})$.
Moreover, $\Lambda^{\mathcal{O}}(\Gamma_{\chi})$ is a maximal $R$-order since there is an isomorphism of commutative rings
 $\Lambda^{\mathcal{O}}(\Gamma_{\chi}) = \mathcal{O}[[\Gamma_{\chi}]] \simeq \mathcal{O}[[T]]$ and $\mathcal{O}[[T]]$ is integrally closed. Uniqueness follows from commutativity of $\mathcal{Q}^{F}(\Gamma_{\chi})$.
\end{proof}

\subsection{Maximal order $e_{\chi}$-components of Iwasawa algebras}
We give criteria for `$e_{\chi}$-components' of Iwasawa algebras of one-dimensional $p$-adic Lie groups to be maximal orders
in the case that $F/\Q_{p}$ is a sufficiently large finite extension. Moreover,
we give an explicit description of such components.

\begin{prop}\label{prop:chi-comp-max-order}
Let  $\chi \in \Irr_{\Q_{p}^{c}}(\mathcal{G})$ and let $\eta$ be an irreducible constituent of
$\res^{\mathcal{G}}_{H} \chi$.
Let $F/\Q_{p}$ be a finite extension over which both characters $\chi$ and $\eta$ have realisations
and let $\mathcal{O}=\mathcal{O}_{F}$ be its ring of integers.
Suppose that $v_{p}(\eta(1))=v_{p}(|H|)$.
Then $e_{\chi} \in \Lambda^{\mathcal{O}}(\mathcal{G})$,
$\zeta(\Lambda^{\mathcal{O}}(\mathcal{G})e_{\chi}) = \Lambda^{\mathcal{O}}(\Gamma_{\chi})$
and there is an isomorphism of $R:=\mathcal{O}[[\Gamma_{0}]]$-orders
\[
\Lambda^{\mathcal{O}}(\mathcal{G})e_{\chi} \simeq M_{\chi(1) \times \chi(1)}(\Lambda^{\mathcal{O}}(\Gamma_{\chi})).
\]
Moreover, these are maximal $R$-orders and as rings are isomorphic to $M_{\chi(1) \times \chi(1)}(\mathcal{O}[[T]])$.
\end{prop}

\begin{proof}
Suppose that $v_{p}(\eta(1))=v_{p}(|H|)$. Then $v_{p}(\eta^{\gamma^{i}}(1)) =v_{p}(|H|)$ for each $i$. This has two consequences.
First, $e(\eta^{\gamma^{i}}) \in \mathcal{O}[H] \subseteq \Lambda^{\mathcal{O}}(\mathcal{G})$ for each $i$
and so the description of $e_{\chi}$ in \eqref{eq:idem-sum} shows that $e_{\chi} \in \Lambda^{\mathcal{O}}(\mathcal{G})$.
Second, since $\eta$ can be realised over $F$,
a standard result on ``$p$-blocks of defect zero'' (see \cite[Proposition 46 (b)]{MR0450380}, for example) shows that for each $i$ we have
an isomorphism of $\mathcal{O}$-orders
\begin{equation}\label{eq:matrix-order-key-point}
\mathcal{O}[H]e(\eta^{\gamma^{i}}) \simeq M_{\eta(1) \times \eta(1)}(\mathcal{O}).
\end{equation}

Let $M_{\chi}$ be an $\mathcal{O}$-lattice on $V_{\chi}$ that is stable under the action of $\mathcal{G}$.
In particular, $M_{\chi}$ is an $\mathcal{O}[H]$-module and since $e(\eta^{\gamma^{i}}) \in \mathcal{O}[H]$ for each $i$,
we have $M_{\chi} = \oplus_{i=0}^{w_{\chi}-1}e(\eta^{\gamma^{i}})M_{\chi}$.
Now recall the proof of Lemma \ref{lem:c-in-F[H]}.
The element $g^{-1}$ acts on $e(\eta^{\gamma^{i}})M_{\chi}$ for each $i$ and so we see that in fact
\[
c(\eta^{\gamma^{i}}) =
g^{-1} \mid_{e(\eta^{\gamma^{i}})M_{\chi}} \in \mathcal{O}[H]e(\eta^{\gamma^{i}}) \simeq \End_{\mathcal{O}}(e(\eta^{\gamma^{i}})M_{\chi})
\]
where the isomorphism follows from \eqref{eq:matrix-order-key-point}.
Hence $c \in \mathcal{O}[H]e_{\chi} \subseteq \Lambda^{\mathcal{O}}(\mathcal{G})e_{\chi}$ and
since $ge_{\chi} \in  \Lambda^{\mathcal{O}}(\mathcal{G})e_{\chi}$ we have
$\gamma_{\chi} = gc\in \zeta(\Lambda^{\mathcal{O}}(\mathcal{G})e_{\chi})$.
Thus $\Lambda^{\mathcal{O}}(\Gamma_{\chi}) \subseteq \zeta(\Lambda^{\mathcal{O}}(\mathcal{G})e_{\chi})$.
However, $\Lambda^{\mathcal{O}}(\Gamma_{\chi})$ is maximal by Lemma \ref{lem:unique-max-order-in-centre}
and therefore $\zeta(\Lambda^{\mathcal{O}}(\mathcal{G})e_{\chi}) = \Lambda^{\mathcal{O}}(\Gamma_{\chi})$.

Extending scalars in \eqref{eq:matrix-order-key-point} gives an isomorphism of $R$-orders $R[H]e(\eta) \simeq M_{\eta(1) \times \eta(1)}(R)$.
Let $e(\eta) = f_{1} + \cdots + f_{\eta(1)}$ be a decomposition of $e(\eta)$
into a sum of orthogonal indecomposable idempotents of $R[H]e(\eta)$.
Observe that $f_{k,i}:=\gamma^{-i}f_{k}\gamma^{i}$ is also an indecomposable idempotent
for every $0 \leq i < p^{n}$ and $1 \leq k \leq \eta(1)$.
Moreover, each $f_{k,i}$ belongs to $R[H]e(\eta^{\gamma^{i}})$ since
$\gamma^{-i}e(\eta)\gamma^{i}=e(\eta^{\gamma^{i}})$ and $H$ is normal in $\mathcal{G}$.
Thus $e(\eta^{\gamma^{i}}) = f_{1,i} + \cdots + f_{\eta(1),i}$ is a decomposition of $e(\eta^{\gamma^{i}})$
into a sum of orthogonal indecomposable idempotents of $R[H]e(\eta^{\gamma^{i}})$.
Note that $e(\eta) = e(\eta^{\gamma^{i}})$ if and only if  $w_{\chi}$ divides $i$.
Hence $R[H]e_{\chi} = \oplus_{i=0}^{w_{\chi}-1} R[H] e(\eta^{\gamma^{i}})$ is an $R$-suborder of
$\Lambda^{\mathcal{O}}(\mathcal{G})e_{\chi}$ and $e_{\chi} = \sum_{k=1}^{\eta(1)} \sum_{i=0}^{w_{\chi}-1} f_{k,i}$
is a decomposition of $e_{\chi}$ into orthogonal idempotents.
By considering the appropriate `elementary matrices' in $R[H]e(\eta) \simeq  M_{\eta(1) \times \eta(1)}(R)$
together with powers of $\gamma$, it is straightforward to see that there is a subset $S$ of
$(\Lambda^{\mathcal{O}}(\mathcal{G})e_{\chi})^{\times}$
such that each element of $I:=\{ f_{k,i} \}_{k=1,i=0}^{k=\eta(1), i=w_{\chi}-1}$ is a conjugate of $f_{1}=f_{1,0}$
by some element of $S$. Furthermore, $|I|=\eta(1)w_{\chi}=\chi(1)$.
Therefore by \cite[\S 46, Exercise 2]{MR892316} we have a ring isomorphism
\begin{equation}\label{eq:is-matrix-ring-chi1}
\Lambda^{\mathcal{O}}(\mathcal{G})e_{\chi} \simeq M_{\chi(1) \times \chi(1)}(f_{1}\Lambda^{\mathcal{O}}(\mathcal{G})e_{\chi}f_{1}).
\end{equation}

As shown in the proof of \cite[Proposition 6 (2)]{MR2114937},
the dimension of $\mathcal{Q}^{F}(\mathcal{G})e_{\chi}$ over
$\zeta(\mathcal{Q}^{F}(\mathcal{G})e_{\chi}) \simeq \mathcal{Q}^{F}(\Gamma_{\chi})$ is $\chi(1)^{2}$.
Since $e_{\chi}$ is a primitive central idempotent we have $\mathcal{Q}^{F}(\mathcal{G})e_{\chi} \simeq M_{m \times m}(D)$ for some $m$ 
and some skewfield $D$ with $\zeta(D)=\mathcal{Q}^{F}(\Gamma_{\chi})$.
Moreover, $\chi(1)=ms$ where $[D: \mathcal{Q}^{F}(\Gamma_{\chi})]=s^{2}$.
However, \eqref{eq:is-matrix-ring-chi1} shows that $m \geq \chi(1)$ and so in fact $\chi(1)=m$
and
\begin{equation}\label{eq:is-matrix-algebra-chi1}
\mathcal{Q}^{F}(\mathcal{G})e_{\chi} \simeq M_{\chi(1) \times \chi(1)}(\zeta(\mathcal{Q}^{F}(\mathcal{G})e_{\chi})) \simeq M_{\chi(1) \times \chi(1)}(\mathcal{Q}^{F}(\Gamma_{\chi})).
\end{equation}
Combining \eqref{eq:is-matrix-ring-chi1}, \eqref{eq:is-matrix-algebra-chi1} and the fact that
$\zeta(\Lambda^{\mathcal{O}}(\mathcal{G})e_{\chi}) = \Lambda^{\mathcal{O}}(\Gamma_{\chi})$
therefore shows that there are $R$-order isomorphisms
\[
\Lambda^{\mathcal{O}}(\mathcal{G})e_{\chi} \simeq M_{\chi(1) \times \chi(1)}(\zeta(\Lambda^{\mathcal{O}}(\mathcal{G})e_{\chi}))
\simeq M_{\chi(1) \times \chi(1)}(\Lambda^{\mathcal{O}}(\Gamma_{\chi})).
\]

Lemma \ref{lem:unique-max-order-in-centre} and \cite[Theorem 8.7]{MR1972204} show that
$M_{\chi(1) \times \chi(1)}(\Lambda^{\mathcal{O}}(\Gamma_{\chi}))$ is a maximal $R$-order.
The last claim follows from the ring isomorphism $\Lambda^{\mathcal{O}}(\Gamma_{\chi}) = \mathcal{O}[[\Gamma_{\chi}]] \simeq \mathcal{O}[[T]]$.
\end{proof}

\subsection{Central idempotents and Galois actions}\label{subset:idem-gal-actions}
Fix a character $\chi \in \Irr_{\Q_{p}^{c}}(\mathcal{G})$ and let $\eta$ be an irreducible constituent of
$\res^{\mathcal{G}}_{H} \chi$.
We define two fields
\[
K_{\chi} := \Q_{p}(\chi(h) \mid h \in H) \subseteq \Q_{p}(\eta) := \Q_{p}(\eta(h) \mid h \in H)
\]
and remark that the containment is not an equality in general.
Note that $K_{\chi} = K_{\chi \otimes \rho}$ whenever $\rho$ is of type $W$ (i.e.~$\res^{\mathcal{G}}_{H} \rho = 1$).
Since $H$ is normal in $\mathcal{G}$, we have $\Q_{p}(\eta^{g}) = \Q_{p}(\eta)$ for every $g \in \mathcal{G}$ and thus
$\Q_{p}(\eta)$ does not depend on the particular choice $\eta$ of irreducible constituent of $\res^{\mathcal{G}}_{H}\chi$.
We let $\sigma \in \Gal(\Q_{p}^{c}/\Q_{p})$ act on $\chi$ by ${}^{\sigma}\chi(g) = \sigma(\chi(g))$ for all $g \in \mathcal{G}$
and similarly on characters of $H$. Note that the actions on $\res^{\mathcal{G}}_{H}\chi$ and $\eta$ factor through
$\Gal(K_{\chi}/\Q_{p})$ and $\Gal(\Q_{p}(\eta)/\Q_{p})$, respectively.
Moreover, for $\sigma \in \Gal(\Q_{p}^{c}/\Q_{p})$ we have $\sigma(e_{\chi})=e_{({}^{\sigma}\chi)}$,
and $\sigma(e_{\chi})=e_{\chi}$ if and only if $\res^{\mathcal{G}}_{H}\chi = {}^{\sigma}(\res^{\mathcal{G}}_{H}\chi)$.
Hence the action of $\Gal(\Q_{p}^{c}/\Q_{p})$ on $e_{\chi}$ factors through $\Gal(K_{\chi}/\Q_{p})$ and we define
\[
\varepsilon_{\chi} := \sum_{\sigma \in \Gal(K_{\chi} / \Q_{p})} \sigma(e_{\chi}).
\]
We remark that $\varepsilon_{\chi}$ is a primitive central idempotent of $\mathcal{Q}(\mathcal{G})$.
Finally, we define an equivalence relation on $\Irr_{\Q_{p}^{c}}(\mathcal{G})$ as follows:
$\chi \sim \chi'$ if and only if there exists $\sigma \in \Gal(K_{\chi}/\Q_{p})$ such that $\sigma(e_{\chi})=e_{\chi'}$.
Note that the number of equivalence classes is finite and
that we have the decomposition $1 = \sum_{\chi/\sim} \varepsilon_{\chi}$ in $\mathcal{Q}(\mathcal{G})$.

\subsection{Maximal order $\varepsilon_{\chi}$-components of Iwasawa algebras}
We give criteria for `$\varepsilon_{\chi}$-components' of Iwasawa algebras of one-dimensional $p$-adic Lie groups to be maximal orders in the case that $F=\Q_{p}$.
Moreover, we give a somewhat explicit description of such components.

\begin{prop}\label{prop:max-part-of-hybrid-matrix-over-comm-ring}
Let $\chi \in \Irr_{\Q_{p}^{c}}(\mathcal{G})$ and let $\eta$ be an irreducible constituent of
$\res^{\mathcal{G}}_{H} \chi$.
If $v_{p}(\eta(1))=v_{p}(|H|)$ then $\varepsilon_{\chi} \in \Lambda(\mathcal{G})$ and $\varepsilon_{\chi}\Lambda(\mathcal{G})$
is a maximal $\Z_{p}[[\Gamma_{0}]]$-order.
Moreover, $\varepsilon_{\chi}\Lambda(\mathcal{G}) \simeq M_{\chi(1) \times \chi(1)}(S_{\chi})$
for some local integrally closed domain $S_{\chi}$.
\end{prop}

\begin{remark}
The condition that $v_{p}(\eta(1))=v_{p}(|H|)$ is independent of choice of irreducible consistent $\eta$ of $\res^{\mathcal{G}}_{H} \chi$
and thus only depends on $\chi$. Moreover, if the condition holds for $\chi \in \Irr_{\Q_{p}^{c}}(\mathcal{G})$, then
it also holds for ${}^{\sigma}\chi$ for all $\sigma \in \Gal(\Q_{p}^{c}/\Q_{p})$ (recall \S \ref{subset:idem-gal-actions}).
\end{remark}

\begin{proof}[Proof of Proposition \ref{prop:max-part-of-hybrid-matrix-over-comm-ring}]
Assume the notation and hypotheses of Proposition \ref{prop:chi-comp-max-order}.
In particular, Proposition \ref{prop:chi-comp-max-order} shows that $e_{\chi} \in \Lambda^{\mathcal{O}}(\mathcal{G})$.
Hence $\varepsilon_{\chi} \in \mathcal{Q}(\mathcal{G}) \cap \Lambda^{\mathcal{O}}(\mathcal{G}) = \Lambda(\mathcal{G})$.
Now
\[
 \bigoplus_{\sigma \in \Gal(K_{\chi}/\Q_{p})} \Lambda^{\mathcal{O}}(\mathcal{G})\sigma(e_{\chi})
= \Lambda^{\mathcal{O}}(\mathcal{G}) \varepsilon_{\chi}
= \mathcal{O} \otimes_{\Z_{p}} \Lambda(\mathcal{G})  \varepsilon_{\chi}
=\mathcal{O}[[\Gamma_{0}]] \otimes_{\Z_{p}[[\Gamma_{0}]]} \Lambda(\mathcal{G}) \varepsilon_{\chi},
\]
which is a maximal $\mathcal{O}[[\Gamma_{0}]]$-order by Proposition \ref{prop:chi-comp-max-order}.
Hence we have that $\Lambda(\mathcal{G}) \varepsilon_{\chi}$ is a maximal $\Z_{p}[[\Gamma_{0}]]$-order.

Set $\Lambda = \Lambda(\mathcal{G})\varepsilon_{\chi}$ and $S=\zeta(\Lambda)$.
Note that $S$ is contained in $\zeta(\mathcal{Q}(\mathcal{G})\varepsilon_{\chi})$, which is a field
since $\varepsilon_{\chi}$ is a primitive central idempotent of $\mathcal{Q}(\mathcal{G})$.
Thus $S$ is an integral domain.
Moreover, $S$ is semiperfect by \cite[Example (23.3)]{MR1838439} and thus is a finite direct product of local rings by
\cite[Theorem (23.11)]{MR1838439}; therefore $S$ is a local integral domain.
In fact, $S$ is a complete local integral domain by \cite[Proposition 6.5 (ii)]{MR632548}.
Furthermore, $S$ must be integrally closed since $\Lambda(\mathcal{G})\varepsilon_{\chi}$ is maximal.

Now set $\Lambda' = \Lambda^{\mathcal{O}}(\mathcal{G})\varepsilon_{\chi} = \mathcal{O} \otimes_{\Z_{p}} \Lambda$.
It is clear that $\mathcal{O} \otimes_{\Z_{p}} S \subseteq S':=\zeta(\Lambda')$.
A standard argument gives the reverse containment; here we give a slightly modified version of the proof of
\cite[Theorem 7.6]{MR1972204}.
Let $b_{1}, \ldots, b_{m}$ be a $\Z_{p}$-basis of $\mathcal{O}$. Let $x \in S'$.
Then we can write $x= \sum_{i} b_{i} \otimes \lambda_{i}$ for some $\lambda_{i} \in \Lambda$
and in particular, $x$ commutes with $1 \otimes \lambda$ for all $\lambda \in \Lambda$.
Hence $\sum_{i} b_{i} \otimes (\lambda \lambda_{i} - \lambda_{i} \lambda) = 0$
and so $\lambda \lambda_{i} = \lambda_{i} \lambda$ for all $\lambda \in \Lambda$,
that is, $\lambda_{i} \in \zeta(\Lambda)=S$.
Therefore $S' = \mathcal{O} \otimes_{\Z_{p}} S$. This gives
\begin{equation}\label{eqn:extend-by-central-scalars}
\Lambda'  = \mathcal{O} \otimes_{\Z_{p}} \Lambda \simeq (\mathcal{O} \otimes_{\Z_{p}} S) \otimes_{S} \Lambda \simeq S' \otimes_{S} \Lambda.
\end{equation}

Let $S'' :=e_{\chi}S'=\zeta(\Lambda' e_{\chi})$.
By Proposition \ref{prop:chi-comp-max-order} we have $S'' \simeq \mathcal{O}[[\Gamma_{\chi}]]$,
which is a local integrally closed domain.
Let $k$ and $k''$ be the residue fields of $S$ and $S''$, respectively.
Then $k''/k$ is a finite extension of finite fields of characteristic $p$.
By Proposition \ref{prop:chi-comp-max-order} and \eqref{eqn:extend-by-central-scalars} we have
\begin{equation}\label{eqn:defn-of-Gamma}
\Lambda'' := S'' \otimes_{S} \Lambda  \simeq e_{\chi}(S' \otimes_{S} \Lambda)
\simeq e_{\chi}\Lambda' \simeq M_{\chi(1) \times \chi(1)}(S'').
\end{equation}
Now set $\overline{\Lambda} := k \otimes_{S} \Lambda$ and observe that
\[
k'' \otimes_{k} \overline{\Lambda} = k'' \otimes_{k} (k \otimes_{S} \Lambda) \simeq k'' \otimes_{S} \Lambda
\simeq k'' \otimes_{S''} (S'' \otimes_{S} \Lambda) \simeq k'' \otimes_{S''} \Lambda'' \simeq M_{\chi(1) \times \chi(1)}(k'').
\]
Therefore \cite[Theorem 7.18]{MR1972204} shows that $\overline{\Lambda}$ is a separable $k$-algebra.
Now \cite[Theorem 4.7]{MR0121392} shows that $\Lambda$ is separable over $S$.
Hence $\Lambda$ is an Azumaya algebra (i.e.\ separable over its centre) and so defines a class $[\Lambda]$
in the Brauer group $\Br(S)$.
Since $S$ is a complete local ring, \cite[Corollary 6.2]{MR0121392} shows that the natural homomorphism
$\Br(S) \rightarrow \Br(k)$ given by $[X] \mapsto [k \otimes_{S} X]$ is in fact a monomorphism.
However, $\Br(k)$ is trivial by Wedderburn's theorem that every finite division ring is a field.
Therefore $[\Lambda]$ is the trivial element of $\Br(S)$ and so by \cite[Proposition 5.3]{MR0121392} $\Lambda$
is isomorphic to an $S$-algebra of the form $\Hom_{S}(P,P)$ where $P$ is a finitely generated projective faithful $S$-module.
Since $S$ is a local ring, $P$ must be free and so $\Lambda$ is isomorphic to a matrix ring over its centre $S$.
Now \eqref{eqn:defn-of-Gamma} shows that in fact $\Lambda \simeq M_{\chi(1) \times \chi(1)}(S)$.
\end{proof}

The following proposition is an adaptation of \cite[Lemma 2.1]{MR3461042}.

\begin{prop}\label{prop:idempotent-implies-defect-zero}
Let $\chi \in \Irr_{\Q_{p}^{c}}(\mathcal{G})$.
If $\varepsilon_{\chi} \in \Lambda(\mathcal{G})$ then $v_{p}(\eta(1))=v_{p}(|H|)$ for every
irreducible constituent $\eta$ of $\res^{\mathcal{G}}_{H} \chi$.
\end{prop}

\begin{proof}
Fix an irreducible constituent $\eta$ of $\res^{\mathcal{G}}_{H} \chi$ and recall the notation of \S \ref{subset:idem-gal-actions}.
We may write $\varepsilon_{\chi} = \frac{\eta(1)}{|H|} \sum_{h \in H} \beta(h^{-1}) h$, where
\[
\beta := \sum_{\sigma \in \Gal(K_{\chi} / \Q_{p})} \sigma \left( \sum_{i=0}^{w_{\chi}-1} \eta^{\gamma^{i}} \right)
\]
is a $\Q_{p}$-valued character of $H$.
Recall that an element $h \in H$ is said to be $p$-singular if its order is divisible by $p$.
Write $\varepsilon_{\chi} = \sum_{h \in H} a_{h} h$ with $a_{h} \in \Z_{p}$ for $h \in H$.
Then \cite[Proposition 5]{MR1261587} shows that $a_{h}=0$ for every $p$-singular $h \in H$
(alternatively, one can use \cite[Proposition 3]{MR1261587} and that $\varepsilon_{\chi}$ is central).
Hence the character $\beta$ vanishes on $p$-singular elements.
Let $P$ be a Sylow $p$-subgroup of $H$.
Then $\beta$ vanishes on $P- \{1\}$, so the multiplicity of the trivial character of $P$ in the restriction $\res^{H}_{P} \beta$ is
\[
\langle \res^{H}_{P} \beta , 1_{P} \rangle = \beta(1) |P|^{-1} = [K_{\chi} : \Q_{p}] w_{\chi} \eta(1)|P|^{-1}.
\]
Now fix $0 \leq i < w_{\chi}$. Then $P' := \gamma^{i} P \gamma^{-i}$ is also a Sylow $p$-subgroup of $H$. Hence we may write $P' = hPh^{-1}$ for some $h \in H$.
We have
$\langle \res^{H}_{P} \eta^{\gamma^{i}} , 1_{P} \rangle = \langle \res^{H}_{P} \eta^{h} , 1_{P} \rangle = \langle \res^{H}_{P} \eta , 1_{P} \rangle$,
and thus
\[
\langle \res^{H}_{P} \beta, 1_{P} \rangle
= \sum_{\sigma \in \Gal(K_{\chi} / \Q_{p})} \sum_{i=0}^{w_{\chi}-1} \langle ^{\hat{\sigma}}(\res^{H}_{P}\eta^{\gamma^{i}}), 1_{P} \rangle
= [K_{\chi} : \Q_{p}] w_{\chi}  \langle \res^{H}_{P} \eta , 1_{P} \rangle,
\]
where $\hat \sigma \in \Gal(\Q_{p}(\eta) / \Q_{p})$ denotes any lift of $\sigma$.
Therefore $\eta(1) = |P| \langle \res^{H}_{P} \eta , 1_{P} \rangle$, and so $v_{p}(\eta(1))=v_{p}(|H|)$.
\end{proof}

\begin{prop}\label{prop:varep-comp-TFAE}
Let $\chi \in \Irr_{\Q_{p}^{c}}(\mathcal{G})$ and let $\eta$ be an irreducible constituent of
$\res^{\mathcal{G}}_{H} \chi$. Then the following are equivalent:
\begin{enumerate}
\item $v_{p}(\eta(1))=v_{p}(|H|)$,
\item $\varepsilon_{\chi} \in \Lambda(\mathcal{G})$,
\item $\varepsilon_{\chi} \in \Lambda(\mathcal{G})$ and $\varepsilon_{\chi}\Lambda(\mathcal{G})$
is a maximal $\Z_{p}[[\Gamma_{0}]]$-order.
\end{enumerate}
Moreover, if these equivalent conditions hold then
$\varepsilon_{\chi}\Lambda(\mathcal{G}) \simeq M_{\chi(1) \times \chi(1)}(S_{\chi})$
for some local integrally closed domain $S_{\chi}$.
\end{prop}

\begin{proof}
This is the combination of Propositions \ref{prop:max-part-of-hybrid-matrix-over-comm-ring} and \ref{prop:idempotent-implies-defect-zero}.
\end{proof}

\subsection{Hybrid Iwasawa algebras}
Let $p$ be an odd prime and let $\mathcal{G} = H \rtimes \Gamma$ be a one-dimensional $p$-adic Lie group. 
As in \S \ref{subsec:Iwasawa-algebras}, we choose an open subgroup $\Gamma_{0} \leq \Gamma$ that is central in $\mathcal{G}$.
If $N$ is a finite normal subgroup of $\mathcal{G}$ and $e_{N} := |N|^{-1}\sum_{\sigma \in N} \sigma$
is the associated central trace idempotent in the algebra $\mathcal{Q}(\mathcal{G})$,
then there is a ring isomorphism $\Lambda(\mathcal{G})e_{N} \simeq \Lambda(\mathcal{G}/N)$.
In particular, the commutator subgroup $\mathcal{G}'$ of $\mathcal{G}$ is contained in the finite subgroup $H$ and
thus $\Lambda(\mathcal{G})e_{\mathcal{G}'} \simeq \Lambda(\mathcal{G}^{\mathrm{ab}})$ where
$\mathcal{G}^{\mathrm{ab}}:=\mathcal{G}/\mathcal{G}'$ is the maximal abelian quotient of $\mathcal{G}$.

\begin{definition}\label{def:max-comm-hybrid}
Let $N$ be a finite normal subgroup of $\mathcal{G}$.
We say that the Iwasawa algebra $\Lambda(\mathcal{G})$ is \emph{$N$-hybrid}
if (i) $e_{N} \in \Lambda(\mathcal{G})$ (i.e.\ $p \nmid |N|$) and
(ii) $\Lambda(\mathcal{G})(1-e_{N})$ is a maximal $\Z_{p}[[\Gamma_{0}]]$-order in
$\mathcal{Q}(\mathcal{G})(1-e_{N})$.
\end{definition}

\begin{remark}\label{rmk:p-not-divide-H-max-order}
The Iwasawa algebra $\Lambda(\mathcal{G})$ is itself a maximal order if and only if $p$ does not divide $|H|$ if and only if 
$\Lambda(\mathcal{G})$ is $H$-hybrid. Moreover, $\Lambda(\mathcal{G})$ is always $\{ 1 \}$-hybrid.
\end{remark}

\begin{lemma}\label{lem:kernel-criterion}
Let $\chi \in \Irr_{\Q_{p}^{c}}(\mathcal{G})$ and let $\eta$ be an irreducible constituent of $\res^{\mathcal{G}}_{H} \chi$.
Let $N$ be a finite normal subgroup of $\mathcal{G}$.
Then $N$ is in fact a normal subgroup of $H$.
Moreover, $N \leq \ker(\chi)$ if and only if $N \leq \ker(\eta)$.
\end{lemma}

\begin{proof}
The canonical projection $\mathcal{G} \twoheadrightarrow \mathcal{G} / H \simeq \mathbb Z_{p}$ maps $N$ onto $NH /H \simeq N / H\cap N$.
However, the only finite (normal) subgroup of $\mathbb{Z}_{p}$ is the trivial group;
hence $H \cap N = N$ and thus $N$ is a (normal) subgroup of $H$.
As $\chi$ factors through a finite quotient of $\mathcal{G}$, the second claim follows from
Lemma \ref{lem:kernel-basechange} (ii).
\end{proof}

\begin{lemma}\label{lem:exists-chi-irr-constit-eta}
Let $\eta \in \Irr_{\Q_{p}^{c}}(H)$. Then there exists $\chi \in \Irr_{\Q_{p}^{c}}(\mathcal{G})$ such that $\eta$ is an irreducible constituent of
$\res^{\mathcal{G}}_{H} \chi$.
\end{lemma}

\begin{proof}
Let $G=\mathcal{G}/\Gamma_{0} \simeq H \rtimes \Gamma/\Gamma_{0}$ (recall that $\Gamma_{0}$ is central in $\mathcal{G}$).
Let $\overline{\chi}$ be an irreducible constituent of $\ind^{G}_{H} \eta$.
Then by Frobenius reciprocity $\langle \res^{G}_{H} \overline{\chi}, \eta \rangle = \langle \overline{\chi}, \ind_{H}^{G} \eta \rangle \neq 0$.
Now take $\chi = \infl^{\mathcal{G}}_{G} \overline{\chi}$.
\end{proof}

\begin{theorem}\label{thm:hybrid-criterion}
Let $p$ be an odd prime and let $\mathcal{G} = H \rtimes \Gamma$ be a one-dimensional $p$-adic Lie group
with a finite normal subgroup $N$.
Then $N$ is in fact a normal subgroup of $H$.
Moreover, $\Lambda(\mathcal{G}):=\Z_{p}[[\mathcal{G}]]$ is $N$-hybrid if and only if $\Z_{p}[H]$ is $N$-hybrid.
\end{theorem}

\begin{proof}
The first assertion is contained in Lemma \ref{lem:kernel-criterion}.

Suppose that $\Z_{p}[H]$ is $N$-hybrid.
Let $\chi \in \Irr_{\Q_{p}^{c}}(\mathcal{G})$ such that $N \not \leq \ker(\chi)$.
Let $\eta$ be an irreducible constituent of $\res^{\mathcal{G}}_{H} \chi$.
Then $N \not \leq \ker(\eta)$ by Lemma \ref{lem:kernel-criterion} and so
$v_{p}(\eta(1))=v_{p}(|H|)$ by Proposition \ref{prop:hybrid-criterion-groupring}.
Hence $\varepsilon_{\chi} \in \Lambda(\mathcal{G})$ and $\varepsilon_{\chi}\Lambda(\mathcal{G})$
is a maximal $\Z_{p}[[\Gamma_{0}]]$-order by Proposition \ref{prop:varep-comp-TFAE}.
Now by Lemma \ref{lem:kernel-criterion} we have
\begin{equation}\label{eq:1-en-decomp}
 1-e_{N}
= \sum_{\stackrel{\eta \in \Irr_{\Q_{p}^{c}}(H)}{N \not \leq \ker(\eta)}} e(\eta)
= \sum_{\stackrel{\chi \in \Irr_{\Q_{p}^{c}}(\mathcal{G})/\sim}{N \not \leq \ker(\chi)}} \varepsilon_{\chi}
\end{equation}
where $\sim$ denotes the equivalence relation defined in \S \ref{subset:idem-gal-actions}.
Therefore $1-e_{N} \in \Lambda(\mathcal{G})$ and
 $\Lambda(\mathcal{G})(1-e_{N})$ is a maximal $\Z_{p}[[\Gamma_{0}]]$-order
as it is the direct product of such orders. Hence $\Lambda(\mathcal{G})$ is $N$-hybrid.

Suppose conversely that $\Lambda(\mathcal{G})$ is $N$-hybrid.
Let $\eta \in \Irr_{\Q_{p}^{c}}(H)$ such that $N \not \leq \ker(\eta)$.
By Lemma \ref{lem:exists-chi-irr-constit-eta} there exists $\chi \in \Irr_{\Q_{p}^{c}}(\mathcal{G})$ such that $\eta$ is an irreducible
constituent of $\res^{\mathcal{G}}_{H} \chi$ and by Lemma \ref{lem:kernel-criterion}
we have $N \not \leq \ker(\chi)$.
Since $1-e_{N} \in \Lambda(\mathcal{G})$ and $\Lambda(\mathcal{G})(1-e_{N})$
is a maximal $\Z_{p}[[\Gamma_{0}]]$-order, it follows from
\eqref{eq:1-en-decomp} that in particular $\varepsilon_{\chi} \in \Lambda(\mathcal{G})$.
Thus Proposition \ref{prop:varep-comp-TFAE} gives $v_{p}(\eta(1))=v_{p}(|H|)$.
Therefore $\Z_{p}[H]$ is $N$-hybrid by Proposition \ref{prop:hybrid-criterion-groupring}.
\end{proof}

\begin{corollary}\label{cor:hybrid-implies-nr-surjective}
Let $p$ be an odd prime and let $\mathcal{G} = H \rtimes \Gamma$ be a one-dimensional $p$-adic Lie group
with a finite normal subgroup $N$.
If $\Lambda(\mathcal{G}) := \Z_{p}[[\mathcal{G}]]$ is $N$-hybrid then
$\Lambda(\mathcal{G})(1-e_{N})$ is isomorphic to a direct product of matrix rings over integrally closed commutative local rings.
\end{corollary}

\begin{proof}
This follows from Proposition \ref{prop:varep-comp-TFAE} and the proof of Theorem \ref{thm:hybrid-criterion} above.
\end{proof}

\begin{prop}\label{prop:frob-N-iwasawa-hybrid}
Let $p$ be an odd prime and $\mathcal{G} = H \rtimes \Gamma$ be a one-dimensional $p$-adic Lie group.
Suppose that $H$ is a Frobenius group with Frobenius kernel $N$.
If $p$ does not divide $|N|$, then $\Lambda(\mathcal{G}):=\Z_{p}[[\mathcal{G}]]$ is $N$-hybrid.
\end{prop}

\begin{proof}
Let $\gamma$ be a topological generator of $\Gamma$.
Then $N$ is a normal subgroup of $H$ and so $\gamma N \gamma^{-1}$ is also a normal subgroup of $H$
since $\gamma H = H \gamma$. Thus $N$ and $\gamma N \gamma^{-1}$ are normal subgroups of $H$ of equal order and so
Theorem \ref{thm:frob-kernel} (iii) implies that they are in fact equal. Hence $N$ is normal in $\mathcal{G}$.
The claim is now an immediate consequence of Proposition \ref{prop:hybrid-criterion-groupring}
and Theorem \ref{thm:hybrid-criterion}.
\end{proof}

\begin{example}
Let $q$ be a prime power and let $H=\Aff(q)$ be the Frobenius group with Frobenius kernel $N$ defined in Example \ref{ex:affine}.
Let $p$ be an odd prime not dividing $q$ and let $\mathcal{G} = H \rtimes \Gamma$ (any choice of semidirect product).
Then $p$ does not divide $|N|$ and so
$\Lambda(\mathcal{G}) := \Z_{p}[[\mathcal{G}]]$ is $N$-hybrid  by Proposition \ref{prop:frob-N-iwasawa-hybrid}.
We consider its structure in more detail.
Let $\eta$ be the unique non-linear irreducible character of $H$.
Choose $\chi \in \Irr_{\Q_{p}^{c}}(\mathcal{G})$
such that $\eta$ appears as an irreducible constituent of $\res^{\mathcal{G}}_{H} \chi$ (this is possible by Lemma \ref{lem:exists-chi-irr-constit-eta}.)
As $\eta$ is the only non-linear irreducible character of $H$, we must have $\eta^{g} = \eta$ for every $g \in \mathcal{G}$,
i.e., $St(\eta) = \mathcal{G}$. Consequently, we have $w_{\chi} = 1$ and thus $\chi(1) = \eta(1) = q-1$.
Since both $\eta$ and $\chi$ have realisations over $\Q$ and hence over $\Q_{p}$,
applying Proposition \ref{prop:chi-comp-max-order} therefore shows that there is a ring isomorphism
$\Lambda(\mathcal{G}) \simeq \Z_{p}[[C_{q-1} \rtimes \Gamma]] \oplus M_{(q-1) \times (q-1)}(\Z_{p}[[T]])$.
\end{example}

\begin{example}\label{ex:S4-V4}
Let $p=3$, $H=S_{4}$ and $N=V_{4}$.
Recall from Example \ref{ex:S4-A4-V4} that $\Z_{3}[S_{4}]$ is $V_{4}$-hybrid and $S_{4}$ is \emph{not} a Frobenius group.
Let $\mathcal{G} = H \rtimes \Gamma$ (any choice of semidirect product).
As each automorphism of $S_{4}$ is inner, $N$ is normal in $\mathcal{G}$
and so Theorem \ref{thm:hybrid-criterion} shows that $\Lambda(\mathcal{G}):=\Z_{3}[[\mathcal{G}]]$ is $N$-hybrid.
We consider its structure in more detail.
Note that $H/N \simeq S_{3}$ and the only two complex irreducible characters $\eta$ and $\eta'$ of $H$
not inflated from characters of $S_{3}$ are of degree $3$ and have realisations over $\Q$ and hence over $\Q_{3}$.
Choose characters $\chi, \chi' \in \Irr_{\Q_{3}^{c}}(\mathcal{G})$ with the property that $\eta$ and $\eta'$ appear as irreducible constituents of $\res^{\mathcal{G}}_{H} \chi$ and $\res^{\mathcal{G}}_{H} \chi'$, respectively.
Again, as each automorphism of $S_{4}$ is inner, we have
$St(\eta) = St(\eta') = \mathcal{G}$ and thus $w_{\chi} = w_{\chi'} = 1$.
Therefore Proposition \ref{prop:chi-comp-max-order} yields a ring isomorphism
$\Lambda(\mathcal{G}) \simeq \Z_{3}[[S_{3} \rtimes \Gamma]] \oplus M_{3 \times 3}(\Z_{3}[[T]]) \oplus  M_{3 \times 3}(\Z_{3}[[T']]))$.
\end{example}

\subsection{Iwasawa algebras and commutator subgroups}
Let $p$ be prime (not necessarily odd) and let $\mathcal{G}=H \rtimes \Gamma$ be a one-dimensional $p$-adic Lie group.

\begin{prop}[{\cite[Proposition 4.5]{MR3092262}}]\label{prop:niceIwasawa-algebras}
The Iwasawa algebra $\Lambda(\mathcal{G})=\Z_{p}[[\mathcal{G}]]$ is a direct product of matrix rings over commutative rings
if and only if $p$ does not divide the order of the commutator subgroup $\mathcal{G}'$ of $\mathcal{G}$.
\end{prop}

\begin{corollary}\label{cor:no-skewfields}
If $p$ does not divide the order of $\mathcal{G}'$ then no skewfields appear in the Wedderburn decomposition of $\mathcal{Q}(\mathcal{G})$.
\end{corollary}

\begin{remark}
Note that $\mathcal{G}'$ is a normal subgroup of $H$.
If $\Lambda(\mathcal{G})$ is $\mathcal{G}'$-hybrid then $p$ does not divide the order of $\mathcal{G}'$.
However, the converse does not hold in general.
\end{remark}

\subsection{Hybrid algebras in Iwasawa theory}\label{subsec:algebras-in-Iwasawa-thy}
Let $p$ be an odd prime.
We denote the cyclotomic $\Z_{p}$-extension of a number field $K$ by $K_{\infty}$ and let $K_{m}$ be its $m$th layer.
We put $\Gamma_{K} := \Gal(K_{\infty}/K)$
and choose a topological generator $\gamma_{K}$. In Iwasawa theory one is often concerned with the
following situation. Let $L/K$ be a finite Galois extension of number fields with Galois group $G$.
We put $H := \Gal(L_{\infty}/K_{\infty})$ and $\mathcal{G} := \Gal(L_{\infty}/K)$.
Then $H$ naturally identifies with a normal subgroup of $G$ and $G/H$ is cyclic of $p$-power order
(the field $L^{H}$ equals $L \cap K_{\infty}$ and thus identifies with $K_{m}$ for some $m< \infty$).
As in \S \ref{subsec:Iwasawa-algebras}, we obtain a semidirect product 
$\mathcal{G} = H \rtimes \Gamma$ where $\Gamma \leq \mathcal{G}$ and $\Gamma \simeq \Gamma_{K} \simeq \Z_{p}$.

\begin{prop}\label{prop:hybrid-codescent}
Keep the above notation and suppose that $\Z_{p}[G]$ is $N$-hybrid.
Then $N$ naturally identifies with a normal subgroup of $\mathcal{G}$,
which is also a normal subgroup of $H$.
Moreover, both $\Z_{p}[H]$ and $\Lambda(\mathcal{G}):=\Z_{p}[[\mathcal{G}]]$ are also $N$-hybrid.
\end{prop}

\begin{proof}
Let $F = L^{N}$.
As $e_{N}$ lies in $\Z_{p}[G]$, we have that $p \nmid |N|$.
Hence $N$ naturally identifies with $\Gal(L_{\infty}/F_{\infty})$, which is a normal subgroup of both 
$H$ and $\mathcal{G}$
since $F_{\infty}/K$ is a Galois extension.
Thus $\Z_{p}[H]$ is $N$-hybrid by Proposition \ref{prop:hybrid-basechange-down} and so
$\Lambda(\mathcal{G})$ is also $N$-hybrid by Theorem \ref{thm:hybrid-criterion}.
\end{proof}

\begin{example} \label{ex:S4-V4-II}
Let $p=3$ and suppose that $G = \Gal(L/K) \simeq S_{4}$.
Then we also have $\Gal(L_{\infty}/K_{\infty}) \simeq S_{4}$, since $S_{4}$ has no abelian quotient of $3$-power order.
As a consequence, we have $\mathcal{G} = \Gal(L_{\infty}/K) \simeq S_{4} \times \Gamma_K$.
Using the notation of Example \ref{ex:S4-V4}, this yields a ring isomorphism
$\Lambda(\mathcal{G}) \simeq \Z_{3}[[S_{3} \times \Gamma_{K}]]
\oplus M_{3 \times 3}(\Z_{3}[[T]]) \oplus  M_{3 \times 3}(\Z_{3}[[T']]))$.
\end{example}

\begin{example}
Assume that $G = N \rtimes V$ is a Frobenius group and that $p \nmid |N|$.
Then by Proposition \ref{prop:frob-N-hybrid} the group ring $\Z_{p}[G]$ is $N$-hybrid.
It is straightforward to check that $H = \Gal(L_{\infty}/K_{\infty}) \simeq N \rtimes U$ is a Frobenius group with $U \leq V$.
Let $F = L^{N}$.
If we assume that $V$ is abelian, then $\Gal(F_{\infty}/K)$ is also abelian and so is isomorphic to $\Gamma \times U$ for some
choice of $\Gamma \simeq \Z_{p}$. Thus we have an isomorphism
$\Lambda(\mathcal{G}) \simeq \Lambda(\Gamma \times U) \oplus (1-e_{N}) \mathfrak{M}(\mathcal{G})$
where $\mathfrak{M}(\mathcal{G})$ is a maximal order containing $\Lambda(\mathcal{G})$.
\end{example}

\section{The equivariant Iwasawa main conjecture} \label{sec:EIMC}

\subsection{Algebraic $K$-theory}\label{subsec:K-theory}
Let $R$ be a noetherian integral domain with field of fractions $E$.
Let $A$ be a finite-dimensional semisimple $E$-algebra and let $\mathfrak{A}$ be an $R$-order in $A$.
Let $\PMod(\mathfrak{A})$ denote the category of finitely generated projective (left) $\mathfrak{A}$-modules.
We write $K_{0}(\mathfrak{A})$ for the Grothendieck group of $\PMod(\mathfrak{A})$ (see \cite[\S 38]{MR892316})
and $K_{1}(\mathfrak{A})$ for the Whitehead group (see \cite[\S 40]{MR892316}).
Let $K_{0}(\mathfrak{A}, A)$ denote the relative algebraic $K$-group associated to the ring homomorphism
$\mathfrak{A} \hookrightarrow A$.
We recall that $K_{0}(\mathfrak{A}, A)$ is an abelian group with generators $[X,g,Y]$ where
$X$ and $Y$ are finitely generated projective $\mathfrak{A}$-modules
and $g:E \otimes_{R} X \rightarrow E \otimes_{R} Y$ is an isomorphism of $A$-modules;
for a full description in terms of generators and relations, we refer the reader to \cite[p.\ 215]{MR0245634}.
Moreover, there is a long exact sequence of relative $K$-theory (see \cite[Chapter 15]{MR0245634})
\begin{equation}\label{eqn:long-exact-seq}
K_{1}(\mathfrak{A}) \longrightarrow K_{1}(A) \stackrel{\partial}{\longrightarrow} K_{0}(\mathfrak{A}, A)
\stackrel{\rho}{\longrightarrow} K_{0}(\mathfrak{A}) \longrightarrow K_{0}(A).
\end{equation}
The reduced norm map $\nr = \nr_{A}: A \rightarrow \zeta(A)$ is defined componentwise on the Wedderburn decomposition of $A$
and extends to matrix rings over $A$ (see \cite[\S 7D]{MR632548}); thus
it induces a map $K_{1}(A) \longrightarrow \zeta(A)^{\times}$, which we also denote by $\nr$.

Let $\mathcal C^{b} (\PMod (\mathfrak{A}))$ be the category of bounded complexes of finitely generated projective $\mathfrak{A}$-modules.
Then $K_{0}(\mathfrak{A}, A)$ identifies with the Grothendieck group whose generators are $[C^{\bullet}]$, where $C^{\bullet}$
is an object of the category $\mathcal C^{b}\tor(\PMod(\mathfrak{A}))$ of bounded complexes of finitely generated projective $\mathfrak{A}$-modules whose 
cohomology modules are $R$-torsion, and the relations are as follows: $[C^{\bullet}] = 0$ if $C^{\bullet}$ is acyclic, and
$[C_{2}^{\bullet}] = [C_{1}^{\bullet}] + [C_{3}^{\bullet}]$ for every short exact sequence
\[
0 \longrightarrow C_{1}^{\bullet} \longrightarrow C_{2}^{\bullet} \longrightarrow C_{3}^{\bullet} \longrightarrow 0
\]
in $\mathcal C^{b}\tor(\PMod(\mathfrak{A}))$ (see \cite[Chapter 2]{MR3076731} or \cite[\S 2]{MR3068893}, for example).

Let $\mathcal{D} (\mathfrak{A})$ be the derived category of $\mathfrak{A}$-modules.
A complex of $\mathfrak{A}$-modules is said to be perfect if it is
isomorphic in $\mathcal{D} (\mathfrak{A})$ to an element of $\mathcal C^b(\PMod (\mathfrak{A}))$.
We denote the full triangulated subcategory of
$\mathcal{D} (\mathfrak{A})$ comprising perfect complexes by $\mathcal{D}^{\perf} (\mathfrak{A})$,
and the full triangulated subcategory
comprising perfect complexes whose cohomology modules are $R$-torsion by $\mathcal{D}^{\perf}\tor (\mathfrak{A})$.
Then any object of $\mathcal{D}^{\perf}\tor (\mathfrak{A})$ defines an element in $K_{0}(\mathfrak{A}, A)$.

We now specialise to the situation of \S \ref{subsec:Iwasawa-algebras}.
Let $p$ be an odd prime and let $\mathcal{G} = H \rtimes \Gamma$ be a one-dimensional $p$-adic Lie group.
Let $A = \mathcal{Q}(\mathcal{G})$, $\mathfrak{A} = \Lambda(\mathcal{G})=\Z_{p}[[\mathcal{G}]]$ and $R=\Z_{p}[[\Gamma_{0}]]$,
where $\Gamma_{0}$ is an open subgroup of $\Gamma$ that is central in $\mathcal{G}$.
Then \cite[Corollary 3.8]{MR3034286}  (take $\mathcal{O}=\Z_{p}$ and $G=\mathcal{G}$ and note that
$\mathcal{O}[[G]]_{S^{*}}=\mathcal{Q}(\mathcal{G})$ since $\mathcal{G}$ is one-dimensional)
shows that the map $\partial$ in \eqref{eqn:long-exact-seq}
is surjective
(one can also give a slight modification of the proof of either \cite[Proposition 3.4]{MR2217048} or \cite[Lemma 1.5]{MR2819672});
thus the sequence
\begin{equation}\label{eqn:Iwasawa-K-sequence}
K_{1}(\Lambda(\mathcal{G})) \longrightarrow K_{1}(\mathcal{Q}(\mathcal{G})) \stackrel{\partial}{\longrightarrow}
K_{0}(\Lambda(\mathcal{G}),\mathcal{Q}(\mathcal{G})) \longrightarrow 0
\end{equation}
is exact.

\subsection{Admissible extensions and the $\mu=0$ hypothesis}\label{subsec:admiss-and-mu}

We specialise the definition of admissible $p$-adic Lie extension given in the introduction to the one-dimensional case.

\begin{definition}\label{def:one-dim-adm}
Let $p$ be an odd prime and let $K$ be a totally real number field.
An admissible one-dimensional $p$-adic Lie extension $\mathcal{L}$ of $K$ is a Galois extension $\mathcal{L}$ of $K$
such that (i) $\mathcal{L}$ is totally real,
(ii) $\mathcal{L}$ contains the cyclotomic $\Z_{p}$-extension $K_{\infty}$ of $K$, and
(iii) $[\mathcal{L} : K_{\infty}]$ is finite.
\end{definition}

Let $\mathcal{L}/K$ be an admissible one-dimensional $p$-adic Lie extension with Galois group $\mathcal{G}$.
Let $H=\Gal(\mathcal{L}/K_{\infty})$ and let $\Gamma_{K}=\Gal(K_{\infty}/K)$.
As in \S \ref{subsec:Iwasawa-algebras}, we obtain a semidirect product 
$\mathcal{G} = H \rtimes \Gamma$ where $\Gamma \leq \mathcal{G}$ and $\Gamma \simeq \Gamma_{K} \simeq \Z_{p}$, 
and we choose an open subgroup $\Gamma_{0} \leq \Gamma$ that is central in $\mathcal{G}$. 
 
Let $S_{\infty}$ be the set of archimedean places of $K$ and let $S_{p}$ be the set of places of $K$ above $p$.
Let $S_{\ram}=S_{\ram}(\mathcal{L}/K)$ be the (finite) set of places of $K$ that ramify in $\mathcal{L}/K$;
note that  $S_{p} \subseteq S_{\ram}$.
Let $S$ be a finite set of places of $K$ containing $S_{\ram} \cup S_{\infty}$.
Let $M_{S}^{\ab}(p)$ be the maximal abelian pro-$p$-extension of 
$\mathcal{L}$ unramified outside $S$ and let $X_{S}=\Gal(M_{S}^{\ab}(p)/\mathcal{L})$. 
As usual $\mathcal{G}$ acts on $X_{S}$ by $g \cdot x = \tilde{g}x\tilde{g}^{-1}$, 
where $g \in \mathcal{G}$, and $\tilde{g}$ is any lift of $g$ to $\Gal(M_{S}^{\ab}(p)/K)$. 
This action extends to a left action of $\Lambda(\mathcal{G})$ on $X_{S}$.
Since $\mathcal{L}$ is totally real, a result of Iwasawa \cite{MR0349627} shows that 
$X_{S}$ is finitely generated and torsion as a $\Lambda(\Gamma_{0})$-module.

\begin{definition}\label{def:mu=0-hypothesis}
We say that $\mathcal{L}/K$ satisfies the $\mu=0$ hypothesis if $X_{S}$ is finitely generated as a $\Z_{p}$-module.
\end{definition}

\begin{remark}\label{rmk:mu=0}
The classical Iwasawa $\mu=0$ conjecture (at $p$) is the assertion that for every number field $F$, the
Galois group of the maximal unramified abelian $p$-extension of $F_{\infty}$ is a finitely generated $\Z_{p}$-module.
This conjecture has been proven by Ferrero and Washington \cite{MR528968} in the case that $F/\Q$ is abelian.
Now let $\mathcal{L}/K$ be an admissible one-dimensional $p$-adic Lie extension
and let $L$ be a finite Galois extension of $K$ such that $L_{\infty}=\mathcal{L}$.
Let $E$ be an intermediate field of $L/K$ such that $L/E$ is of $p$-power degree.
Then \cite[Theorem 11.3.8]{MR2392026} says that $\mathcal{L}/K$
satisfies the $\mu=0$ hypothesis if and only if $E_{\infty}/K$ does.
Finally, let $\zeta_{p}$ denote a primitive $p$th root of unity.
Then by \cite[Corollary 11.4.4]{MR2392026} Iwasawa's conjecture for $E(\zeta_{p})$ implies the $\mu=0$ hypothesis for 
$E_{\infty}(\zeta_{p})^{+}/K$ and thus for $E_{\infty}/K$ and $\mathcal{L}/K$.
\end{remark}

\subsection{A reformulation of the equivariant Iwasawa main conjecture} \label{subsec:EIMC-reformulation}
We give a slight reformulation of the equivariant Iwasawa main conjecture for totally real fields.

Let $\mathcal{L}/K$ be an admissible one-dimensional $p$-adic Lie extension.
We assume the notation and setting of \S \ref{subsec:admiss-and-mu}.
However, we do \emph{not} assume the $\mu=0$ hypothesis for $\mathcal{L}/K$ except where explicitly stated.
Let $C_{S}^{\bullet}(\mathcal{L}/K)$ be the canonical complex
\[
    C_{S}^{\bullet}(\mathcal{L}/K) := R\Hom(R\Gamma_{\et}(\Spec(\mathcal{O}_{\mathcal{L},S}), \Q_{p} / \Z_{p}), \Q_{p} / \Z_{p}).
\]
Here, $\mathcal{O}_{\mathcal{L},S}$ denotes the ring of integers $\mathcal{O}_{\mathcal{L}}$ in $\mathcal{L}$ localised at all primes above those in $S$
and
$\Q_{p} / \Z_{p}$ denotes the constant sheaf of the abelian group $\Q_{p} / \Z_{p}$ on the \'{e}tale site
of $\Spec(\mathcal{O}_{\mathcal{L},S})$.
The only non-trivial cohomology groups occur in degree $-1$ and $0$ and we have
\[
H^{-1}(C_{S}^{\bullet}(\mathcal{L}/K)) \simeq X_{S}, \qquad H^{0}(C_{S}^{\bullet}(\mathcal{L}/K)) \simeq \Z_{p}.
\]
It follows from \cite[Proposition 1.6.5]{MR2276851} that $C_{S}^{\bullet}(\mathcal{L}/K)$ belongs to $\mathcal{D}^{\perf}\tor(\Lambda(\mathcal{G}))$.
In particular, $C_{S}^{\bullet}(\mathcal{L}/K)$ defines a class $[C_{S}^{\bullet}(\mathcal{L}/K)]$ in $K_{0}(\Lambda(\mathcal{G}), \mathcal{Q}(\mathcal{G}))$.
Note that $C_{S}^{\bullet}(\mathcal{L}/K)$ and the complex used
by Ritter and Weiss (as constructed in \cite{MR2114937}) become isomorphic in $\mathcal{D}(\Lambda(\mathcal{G}))$ by
\cite[Theorem 2.4]{MR3072281} (see also \cite{MR3068897} for more on this topic).
Hence it makes no essential difference which of these complexes we use.

Recall the notation and hypotheses of \S \ref{subsec:idempotents} and \S \ref{subsec:sufficiently-large}.
In particular, $F$ is a sufficiently large finite extension of $\Q_{p}$.
Let $\chi_{\mathrm{cyc}}$ be the $p$-adic cyclotomic character
\[
\chi_{\mathrm{cyc}}: \Gal(\mathcal{L}(\zeta_{p})/K) \longrightarrow \Z_{p}^{\times},
\]
defined by $\sigma(\zeta) = \zeta^{\chi_{\mathrm{cyc}}(\sigma)}$ for any $\sigma \in \Gal(\mathcal{L}(\zeta_{p})/K)$ and any $p$-power root of unity $\zeta$.
Let $\omega$ and $\kappa$ denote the composition of $\chi_{\mathrm{cyc}}$ with the projections onto the first and second factors of the canonical decomposition $\Z_{p}^{\times} = \mu_{p-1} \times (1+p\Z_{p})$, respectively;
thus $\omega$ is the Teichm\"{u}ller character.
We note that $\kappa$ factors through $\Gamma_{K}$ 
(and thus also through $\mathcal{G}$) and by abuse of notation we also 
use $\kappa$ to denote the associated maps with these domains.
We put $u := \kappa(\gamma_{K})$.
For $r \in \N_{0}$ divisible by $p-1$ 
(or more generally divisible by the degree $[\mathcal{L}(\zeta_{p}) : \mathcal{L}]$), 
up to the natural inclusion map of codomains, 
we have $\chi_{\mathrm{cyc}}^{r}=\kappa^{r}$. 

Following \cite[Proposition 6]{MR2114937}, we define a map
\[
j_{\chi}: \zeta(\mathcal{Q}^{F} (\mathcal{G})) \twoheadrightarrow \zeta(\mathcal{Q}^{F} (\mathcal{G})e_{\chi}) \simeq \mathcal{Q}^{F}(\Gamma_{\chi}) \rightarrow  \mathcal{Q}^{F}(\Gamma_{K}),
\]
where the last arrow is induced by mapping $\gamma_{\chi}$ to $\gamma_{K}^{w_{\chi}}$.
It follows from op.\ cit.\ that $j_{\chi}$ is independent of the choice of $\gamma_{K}$ and that 
for every matrix $\Theta \in M_{n \times n} (\mathcal{Q}(\mathcal{G}))$ we have
\begin{equation*} \label{eqn:jchi-det}
j_{\chi} (\nr(\Theta)) = \mathrm{det}_{\mathcal{Q}^{F}(\Gamma_{K})} (\Theta \mid \Hom_{F[H]}(V_{\chi},  \mathcal{Q}^{F}(\mathcal{G})^n)).
\end{equation*}
Here, $\Theta$ acts on $f \in \Hom_{F[H]}(V_{\chi},  \mathcal{Q}^{F}(\mathcal{G})^{n})$ via right multiplication,
and $\gamma_{K}$ acts on the left via $(\gamma_{K} f)(v) = \gamma \cdot f(\gamma^{-1} v)$ for all $v \in V_{\chi}$,
where $\gamma$ is the unique lift of $\gamma_{K}$ to $\Gamma \leq \mathcal{G}$.
Hence the map
\begin{eqnarray*}
\Det(~)(\chi): K_{1}(\mathcal{Q}(\mathcal{G})) & \rightarrow & \mathcal{Q}^{F}(\Gamma_{K})^{\times} \\
 {[P,\alpha]}& \mapsto & \mathrm{det}_{\mathcal{Q}^{F}(\Gamma_{K})} (\alpha \mid \Hom_{F[H]}(V_{\chi},  F \otimes_{\Q_{p}} P)),
\end{eqnarray*}
where $P$ is a projective $\mathcal{Q}(\mathcal{G})$-module and $\alpha$ a $\mathcal{Q}(\mathcal{G})$-automorphism of $P$, is just $j_{\chi} \circ \nr$ (see \cite[\S 3, p.558]{MR2114937}).
If $\rho$ is a character of $\mathcal{G}$ of type $W$ (i.e.~$\res^{\mathcal{G}}_H \rho = 1$)
then we denote by
$\rho^{\sharp}$ the automorphism of the field $\mathcal{Q}^{c}(\Gamma_{K})$ induced by
$\rho^{\sharp}(\gamma_{K}) = \rho(\gamma_{K}) \gamma_{K}$. 
Moreover, we denote the additive group generated by all $\Q_{p}^{c}$-valued
characters of $\mathcal{G}$ with open kernel by $R_p(\mathcal{G})$; finally, 
$\Hom^{\ast}(R_{p}( \mathcal{G}), \mathcal{Q}^{c}(\Gamma_{K})^{\times})$
is the group of all homomorphisms 
$f: R_p(\mathcal{G}) \rightarrow \mathcal{Q}^{c}(\Gamma_{K})\mal$ satisfying
\[
\begin{array}{ll}
f(\chi \otimes \rho) = \rho^{\sharp}(f(\chi)) & \mbox{ for all characters } \rho \mbox{ of type } W \mbox{ and}\\
f({}^{\sigma}\chi) = \sigma(f(\chi)) & \mbox{ for all Galois automorphisms } \sigma \in \Gal(\Q_{p}^{c}/\Q_{p}).
\end{array}
\]
By \cite[Proof of Theorem 8]{MR2114937} we have an isomorphism
\begin{eqnarray*}
\zeta(\mathcal{Q}(\mathcal{G}))\mal & \simeq & 
\Hom^{\ast}(R_{p}(\mathcal{G}), \mathcal{Q}^{c}(\Gamma_{K})^{\times})\\
x & \mapsto & [\chi \mapsto j_{\chi}(x)].
\end{eqnarray*}
By \cite[Theorem 8]{MR2114937} the map $\Theta \mapsto [\chi \mapsto \Det(\Theta)(\chi)]$
defines a homomorphism
\[
\Det: K_{1}(\mathcal{Q}(\mathcal{G})) \rightarrow \Hom^{\ast}(R_p(\mathcal{G}), \mathcal{Q}^{c}(\Gamma_{K})\mal)
\]
such that we obtain a commutative triangle
\begin{equation} \label{eqn:Det_triangle}
\xymatrix{
& K_{1}(\mathcal{Q}(\mathcal{G})) \ar[dl]_{\nr} \ar[dr]^{\Det} &\\
{\zeta(\mathcal{Q}(\mathcal{G}))^{\times}} \ar[rr]^{\sim} & & {\Hom^{\ast}(R_p( \mathcal{G}), \mathcal{Q}^{c}(\Gamma_{K})^{\times})}.}
\end{equation}

Each topological generator $\gamma_{K}$ of  $\Gamma_{K}$ permits the definition of a 
power series $G_{\chi,S}(T) \in \Q_{p}^{c} \otimes_{\Q_{p}} Quot(\Z_{p}[[T]])$ 
by starting out from the Deligne-Ribet power series for linear characters of open subgroups 
of $\mathcal{G}$ (see \cite{MR579702}; also see \cite{ MR525346, MR524276}) 
and then extending to the general case by using Brauer induction (see \cite{MR692344}).
One then has an equality
\[
L_{p,S}(1-s,\chi) = \frac{G_{\chi,S}(u^s-1)}{H_{\chi}(u^s-1)},
\]
where $L_{p,S}(s,\chi)$ denotes the `$S$-truncated $p$-adic Artin $L$-function' attached to $\chi$ constructed by Greenberg \cite{MR692344},
and where, for irreducible $\chi$, one has
\[
H_{\chi}(T) = \left\{\begin{array}{ll} \chi(\gamma_{K})(1+T)-1 & \mbox{ if }  H \subseteq \ker \chi\\
1 & \mbox{ otherwise.}  \end{array}\right.
\]
Now \cite[Proposition 11]{MR2114937} implies that
\[
L_{K,S} : \chi \mapsto \frac{G_{\chi,S}(\gamma_{K}-1)}{H_{\chi}(\gamma_{K}-1)}
\]
is independent of the topological generator $\gamma_{K}$ and lies in 
$\Hom^{\ast}(R_{p}( \mathcal{G}), \mathcal{Q}^{c}(\Gamma_{K})^{\times})$.
Diagram \eqref{eqn:Det_triangle} implies that there is a unique element 
$\Phi_{S} = \Phi_{S}(\mathcal{L}/K) \in \zeta(\mathcal{Q}(\mathcal{G}))^{\times}$
such that
\[
j_{\chi}(\Phi_{S}) = L_{K,S}(\chi)
\]
for every $\chi \in \Irr_{\Q_{p}^{c}}(\mathcal{G})$.
It is now clear that the following is a reformulation of the EIMC without its uniqueness statement.

\begin{conj}[EIMC]\label{conj:EIMC}
There exists $\zeta_{S} \in K_{1}(\mathcal{Q}(\mathcal{G}))$ such that $\partial(\zeta_{S}) = -[C_{S}^{\bullet}(\mathcal{L}/K)]$
and $\nr(\zeta_{S}) = \Phi_{S}$.
\end{conj}

It can be shown that the truth of Conjecture \ref{conj:EIMC} is independent of the choice of $S$, 
provided $S$ is finite and contains $S_{\ram} \cup S_{\infty}$.
The following theorem has been shown independently by Ritter and Weiss \cite{MR2813337} and Kakde \cite{MR3091976}.

\begin{theorem}\label{thm:EIMC-with-mu}
If $\mathcal{L}/K$ satisfies the $\mu=0$ hypothesis then the EIMC holds for $\mathcal{L}/K$.
\end{theorem}

\begin{corollary}\label{cor:EIMC-unconditional}
Let $\mathcal{P}$ be a Sylow $p$-subgroup of $\mathcal{G}$.
If $\mathcal{L}^{\mathcal{P}}/\Q$ is abelian then $\mathcal{P}$ is normal in $\mathcal{G}$ (and thus is unique),
and the EIMC holds for $\mathcal{L}/K$.
\end{corollary}

\begin{proof}
The first claim is clear.
Let $E=\mathcal{L}^{\mathcal{P}}$ and let $L$ be a finite Galois extension of $K$ such that $L_{\infty}=\mathcal{L}$.
Then $L/E$ is a finite Galois extension of $p$-power degree.
Moreover, $E/\Q$ is a (finite) abelian extension by hypothesis and so $E(\zeta_{p})/\Q$ is also abelian.
Therefore the $\mu=0$ hypothesis for $\mathcal{L}/K$ holds by the results discussed in Remark \ref{rmk:mu=0}.
\end{proof}

We shall also consider the EIMC with its uniqueness statement.

\begin{conj}[EIMC with uniqueness]\label{conj:EIMC-unique}
There exists a unique $\zeta_{S} \in K_{1}(\mathcal{Q}(\mathcal{G}))$ such that $\partial(\zeta_{S}) = -[C_{S}^{\bullet}(\mathcal{L}/K)]$
and $\nr(\zeta_{S}) = \Phi_{S}$.
\end{conj}

\begin{remark}\label{rmk:SK1}
Let $SK_{1}(\mathcal{Q}(\mathcal{G})) = \ker(\nr: K_{1}(\mathcal{Q}(\mathcal{G})) \longrightarrow \zeta(\mathcal{Q}(\mathcal{G}))^{\times})$.
If $SK_{1}(\mathcal{Q}(\mathcal{G}))$ vanishes then it is clear that the uniqueness statement of the EIMC follows from
its existence statement.
Moreover, $SK_{1}(\mathcal{Q}(\mathcal{G}))$ vanishes if no skewfields appear in the Wedderburn decomposition of
$\mathcal{Q}(\mathcal{G})$; in particular, this is the case if $\mathcal{G}$ is abelian or, more generally, if $p$ does not
divide the order of the commutator subgroup $\mathcal{G}'$ of $\mathcal{G}$ (see Corollary \ref{cor:no-skewfields}).
As noted in \cite[Remark E]{MR2114937} (also see \cite[Remark 3.5]{MR3294653}), a conjecture of Suslin implies that $SK_{1}(\mathcal{Q}(\mathcal{G}))$ in fact always vanishes.
\end{remark}

\subsection{Relation to the framework of \cite{MR2217048}}
We now discuss Conjecture \ref{conj:EIMC} within the framework of the theory of \cite[\S 3]{MR2217048};
this section may be skipped if the reader is only interested in the formulation of \S \ref{subsec:EIMC-reformulation}.
Let
\[
\pi: \mathcal{G} \rightarrow \GL_{n}(\mathcal{O})
\]
be a continuous homomorphism, where $\mathcal{O}=\mathcal{O}_{F}$ denotes the ring of integers of $F$
and $n$ is some integer greater or equal to $1$.
There is a ring homomorphism
\begin{equation} \label{eqn:first_Phi}
\Phi_{\pi}: \Lambda(\mathcal{G}) \rightarrow M_{n\times n}(\Lambda^{\mathcal{O}}(\Gamma_{K}))
\end{equation}
induced by the continuous group homomorphism
\begin{eqnarray*}
\mathcal{G} & \rightarrow & (M_{n \times n}(\mathcal{O}) \otimes_{\Z_p} \Lambda(\Gamma_{K}))\mal = \GL_{n}(\Lambda^{\mathcal{O}}(\Gamma_{K}))\\
\sigma & \mapsto & \pi(\sigma) \otimes \overline{\sigma},
\end{eqnarray*}
where $\overline{\sigma}$ denotes the image of $\sigma$ in $\mathcal{G} / H = \Gamma_{K}$. 
By \cite[Lemma 3.3]{MR2217048} the
homomorphism \eqref{eqn:first_Phi} extends to a ring homomorphism
\[
\Phi_{\pi}: \mathcal{Q}(\mathcal{G}) \rightarrow M_{n\times n}(\mathcal{Q}^{F}(\Gamma_{K}))
\]
and this in turn induces a homomorphism
\[
\Phi_{\pi}': K_{1}(\mathcal{Q}(\mathcal{G})) \rightarrow 
K_{1}(M_{n\times n}(\mathcal{Q}^{F}(\Gamma_{K}))) = \mathcal{Q}^{F}(\Gamma_{K})\mal.
\]
Let $\aug: \Lambda^{\mathcal{O}}(\Gamma_{K}) \twoheadrightarrow \mathcal{O}$ be the augmentation map and put $\mathfrak{p} = \ker(\aug)$.
Writing $\Lambda^{\mathcal{O}}(\Gamma_{K})_{\mathfrak{p}}$ for the localisation of $\Lambda^{\mathcal{O}}(\Gamma_{K})$ at $\mathfrak{p}$, it is clear that $\aug$ naturally extends to a homomorphism $\aug: \Lambda^{\mathcal{O}}(\Gamma_{K})_{\mathfrak{p}} \rightarrow F$.
One defines an evaluation map
\begin{equation*} \label{eqn:evaluation-map}
\begin{array}{rcl}
\phi: \mathcal{Q}^{F}(\Gamma_{K}) & \rightarrow & F \cup \{\infty\}\\
x & \mapsto & \left\{ \begin{array}{ll} \aug (x) & \mbox{ if } x \in \Lambda^{\mathcal{O}}(\Gamma_{K})_{\mathfrak{p}}\\
\infty & \mbox{ otherwise}. \end{array} \right.
\end{array}
\end{equation*}

For $r \in \Z$ we define maps
\[
j_{\chi}^{r}: \zeta(\mathcal{Q}^{F} (\mathcal{G})) \twoheadrightarrow \zeta(\mathcal{Q}^{F} (\mathcal{G})e_{\chi}) \simeq \mathcal{Q}^{F}(\Gamma_{\chi}) \rightarrow  \mathcal{Q}^{F}(\Gamma_{K}),
\]
where the last arrow is induced by mapping $\gamma_{\chi}$ to 
$(u^{r}\gamma_{K})^{w_{\chi}}$.
Note that $j_{\chi}^{0} = j_{\chi}$.
It is straightforward to show that for $r \in \Z$ we have
\begin{equation*}\label{eq:PhiS-jr-p-adic}
\phi(j_{\chi}^{r}(\Phi_{S})) = L_{p,S}(1-r, \chi). 
\end{equation*}

If $\zeta$ is an element of $K_{1}(\mathcal{Q}(\mathcal{G}))$, we define $\zeta(\pi)$ to be $\phi(\Phi_{\pi}'(\zeta))$.
Conjecture \ref{conj:EIMC} now implies that there is an element $\zeta_{S} \in K_{1}(\mathcal{Q}(\mathcal{G}))$ such that
$\partial(\zeta_{S}) = -[C_{S}^{\bullet}(\mathcal{L}/K)]$ and for each $r \geq 1$ divisible by $p-1$
and every irreducible Artin representation $\pi_{\chi}$ of $\mathcal{G}$ with character $\chi$ we have
\[
\zeta_{S}(\pi_{\chi}\kappa^{r}) = \phi(j_{\chi}^{r}(\Phi_{S})) = L_{p,S}(1-r, \chi),
\]
where the first equality follows from \cite[Lemma 2.3]{MR2822866}.

\subsection{A maximal order variant of the EIMC}

We shall prove the EIMC in many cases in which the $\mu=0$ hypothesis is not known;
in some of these cases we shall also prove the EIMC with uniqueness.

The following key result of Ritter and Weiss can be seen as a `maximal order variant' of Conjecture \ref{conj:EIMC}; 
crucially, it does not require the $\mu=0$ hypothesis.
We assume the setup and notation of \S \ref{subsec:EIMC-reformulation}.

\begin{theorem}\label{thm:EIMC-MaxOrd}
Let $\mathfrak{M}(\mathcal{G})$ be a maximal $\Z_{p}[[\Gamma_{0}]]$-order 
such that $\Lambda(\mathcal{G}) \subseteq \mathfrak{M}(\mathcal{G}) \subseteq \mathcal{Q}(\mathcal{G})$.
Choose $x_{S} \in K_{1}(\mathcal{Q}(\mathcal{G}))$ such that $\partial(x_{S}) = -[C_{S}^{\bullet}(\mathcal{L}/K)]$.
Then $\nr(x_{S})\Phi_{S}^{-1} \in \zeta(\mathfrak{M}(\mathcal{G}))\mal$.
\end{theorem}

\begin{proof}
By \cite[Theorem 16]{MR2114937} we know that
$\Det(x_{S})L_{K,S}^{-1} \in \Hom^{\ast}(R_{p}(\mathcal{G}), \Lambda^{c}(\Gamma_{K})^{\times})$,
where $\Lambda^{c}(\Gamma_{K}) := \Z_{p}^{c} \otimes_{\Z_{p}} \Lambda(\Gamma_{K})$ and $\Z_{p}^{c}$ denotes the integral closure
of $\Z_{p}$ in $\Q_{p}^{c}$.
Moreover, $\Hom^{\ast}(R_{p}(\mathcal{G}), \Lambda^{c}(\Gamma_{K})^{\times})$ identifies with $\zeta(\mathfrak{M}(\mathcal{G}))\mal$
under the isomorphism in diagram \eqref{eqn:Det_triangle} as is explained in \cite[Remark H]{MR2114937}. Thus $\nr(x_{S}) \Phi_{S}^{-1}$ lies in $\zeta(\mathfrak{M}(\mathcal{G}))\mal$.
\end{proof}

\begin{corollary} \label{cor:EIMC-MaxOrd}
Let $\mathfrak{M}(\mathcal{G})$ be a maximal $\Z_{p}[[\Gamma_{0}]]$-order such that 
$\Lambda(\mathcal{G}) \subseteq \mathfrak{M}(\mathcal{G}) \subseteq \mathcal{Q}(\mathcal{G})$
and let $e \in \mathfrak{M}(\mathcal{G})$ be a central idempotent. Suppose that the reduced norm map
\begin{equation} \label{eqn:nr-surjective-hypothesis}
\nr: K_{1}(e \mathfrak{M}(\mathcal{G})) \longrightarrow \zeta(e \mathfrak{M}(\mathcal{G}))\mal
\end{equation}
is surjective.
Then there exists $y_{S} \in K_{1}(e \mathcal{Q}(\mathcal{G}))$ such that $\nr(y_{S}) = e \Phi_{S}$
and $y_{S}$ maps to $[e \mathfrak{M}(\mathcal{G}) \otimes^{\mathbb{L}}_{\Lambda(\mathcal{G})} C_{S}^{\bullet}(\mathcal{L}/K)]$ under
$K_{1}(e \mathcal{Q}(\mathcal{G})) \rightarrow K_{0}(e \mathfrak{M}(\mathcal{G}), e \mathcal{Q}(\mathcal{G}))$.
\end{corollary}

\begin{proof}
Choose $x_{S} \in K_{1}(\mathcal{Q}(\mathcal{G}))$ as in Theorem \ref{thm:EIMC-MaxOrd}. By assumption, there is
$z_{S} \in K_{1}(e \mathfrak{M}(\mathcal{G}))$ such that $\nr(z_{S}) = e \nr(x_{S}) \Phi_{S}^{-1}$.
Let $z'_{S}$ be the image of $z_{S}$ in $K_{1}(e \mathcal{Q}(\mathcal{G}))$.
 Then $y_{S} := e x_{S} (z'_{S})^{-1}$ has the desired properties.
\end{proof}

\begin{remark}\label{rmk:nr-surjective}
It is not clear whether the map \eqref{eqn:nr-surjective-hypothesis} is always surjective. 
However, this map is surjective if no skewfields
occur in the Wedderburn decomposition of $e \mathcal{Q}(\mathcal{G})$, and thus one can always take $e = e_{\mathcal{G}'}$,
where $\mathcal{G'}$ is the commutator subgroup of $\mathcal{G}$ (note that $\mathcal{G}' \leq H$).
If $p$ does not divide the order of $\mathcal{G}'$, then by Corollary \ref{cor:no-skewfields}
one can take an arbitrary $e$ (in particular, $e=1$ is possible).
If $\Lambda(\mathcal{G})$ is $N$-hybrid then Corollary \ref{cor:hybrid-implies-nr-surjective} shows that one can take $e=1-e_{N}$.
\end{remark}

\begin{theorem}\label{thm:EIMC-for-p-not-dividing-ordH}
Let $\mathcal{L}/K$ be an admissible one-dimensional $p$-adic Lie extension with Galois group $\mathcal{G} =  H \rtimes \Gamma$.
If $p \nmid |H|$ then the EIMC with uniqueness holds for $\mathcal{L}/K$.
\end{theorem}

\begin{remark}
Theorem \ref{thm:EIMC-for-p-not-dividing-ordH} does not require the $\mu=0$ hypothesis.
The statement is the same as that of \cite[Example 2]{MR2205173},
except that the proof of this latter result does require the $\mu=0$ hypothesis which is not explicitly stated (see \cite[p.\ 48]{MR2242618}).
\end{remark}

\begin{proof}[Proof of Theorem \ref{thm:EIMC-for-p-not-dividing-ordH}]
Since the commutator subgroup $\mathcal{G}'$ is a subgroup of $H$,
Corollary \ref{cor:no-skewfields} shows that the Wedderburn decomposition of $\mathcal{Q}(\mathcal{G})$ contains no skewfields.
This has two consequences. First, uniqueness follows from Remark \ref{rmk:SK1}.
Second, by Remark \ref{rmk:nr-surjective} the hypotheses of Corollary \ref{cor:EIMC-MaxOrd} are satisfied for every choice of $e$.
However, Remark \ref{rmk:p-not-divide-H-max-order} shows that $\Lambda(\mathcal{G})$ is in fact a maximal order since $p \nmid |H|$.
Therefore the EIMC for $\mathcal{L}/K$ follows from Corollary \ref{cor:EIMC-MaxOrd} with $e=1$.
\end{proof}

\subsection{Functorialities}
Let $\mathcal{L}/K$ be an admissible one-dimensional $p$-adic Lie extension with Galois group $\mathcal{G}$.
Let $N$ be a finite normal subgroup of $\mathcal{G}$ and let $\mathcal{H}$ be an open subgroup of $\mathcal{G}$.
There are canonical maps
    \begin{eqnarray*}
        \quot^{\mathcal{G}}_{\mathcal{G}/N}: & K_{0}(\Lambda(\mathcal{G}), \mathcal{Q}(\mathcal{G})) \longrightarrow &
            K_{0}(\Lambda(\mathcal{G}/N), \mathcal{Q}(\mathcal{G}/N)),\\
            \res^{\mathcal{G}}_{\mathcal{H}}: & K_{0}(\Lambda(\mathcal{G}), \mathcal{Q}(\mathcal{G})) \longrightarrow &
            K_{0}(\Lambda(\mathcal{H}), \mathcal{Q}(\mathcal{H}))
    \end{eqnarray*}
    induced from scalar extension along $\Lambda(\mathcal{G}) \longrightarrow \Lambda(\mathcal{G}/N)$ and restriction of scalars
    along $\Lambda(\mathcal{H}) \hookrightarrow \Lambda(\mathcal{G})$. Similarly, we have maps (see \cite[\S 3]{MR2114937})
    \begin{eqnarray*}
        \quot^{\mathcal{G}}_{\mathcal{G}/N}: & \Hom^{\ast}(R_{p}(\mathcal{G}), \mathcal{Q}^{c}(\Gamma_{K})\mal) \longrightarrow &
            \Hom^{\ast}(R_{p}(\mathcal{G}/N), \mathcal{Q}^{c}(\Gamma_{K})\mal),\\
        \res^{\mathcal{G}}_{\mathcal{H}}: & \Hom^{\ast}(R_{p}(\mathcal{G}), \mathcal{Q}^{c}(\Gamma_{K})\mal) \longrightarrow &
            \Hom^{\ast}(R_{p}(\mathcal{H}), \mathcal{Q}^{c}(\Gamma_{K'})\mal),
    \end{eqnarray*}
    where $K' := \mathcal{L}^{\mathcal{H}}$; here for $f \in \Hom^{\ast}(R_{p}(\mathcal{G}), \mathcal{Q}^{c}(\Gamma_{K})\mal)$
    we have $(\quot^{\mathcal{G}}_{\mathcal{G}/N} f)(\chi) = f(\infl^{\mathcal{G}}_{\mathcal{G}/N} \chi)$ and
    $(\res^{\mathcal{G}}_{\mathcal{H}} f)(\chi') = f(\ind^{\mathcal{G}}_{\mathcal{H}} \chi')$ for $\chi \in R_{p}(\mathcal{G}/N)$
    and $\chi' \in R_{p}(\mathcal{H})$. Then diagram \eqref{eqn:Det_triangle} induces canonical maps
    \begin{eqnarray*}
        \quot^{\mathcal{G}}_{\mathcal{G}/N}: & \zeta(\mathcal{Q}(\mathcal{G}))\mal \longrightarrow &
            \zeta(\mathcal{Q}(\mathcal{G}/N))\mal,\\
        \res^{\mathcal{G}}_{\mathcal{H}}: & \zeta(\mathcal{Q}(\mathcal{G}))^{\times} \longrightarrow &
            \zeta(\mathcal{Q}(\mathcal{H}))\mal.
    \end{eqnarray*}
The first map is easily seen to be induced by the canonical projection $\mathcal{G} \twoheadrightarrow \mathcal{G}/N$.

The following proposition is an obvious reformulation of \cite[Proposition 12]{MR2114937}
(note that the proof of \cite[Proposition 12 (1)(a)]{MR2114937} uses a result which assumes Leopoldt's conjecture;
a direct proof without this assumption is given in \cite[Appendix]{MR2813337}); also see
\cite[Proposition 1.6.5]{MR2276851}.

\begin{prop} \label{prop:funtorialities}
Let $\mathcal{L}/K$ be an admissible one-dimensional $p$-adic Lie extension with Galois group $\mathcal{G}$.
Then the following statements hold.
\begin{enumerate}
\item
Let $N$ be a finite normal subgroup of $\mathcal{G}$ and put $\mathcal{L}' := \mathcal{L}^{N}$.
Then
\[
\quot^{\mathcal{G}}_{\mathcal{G}/N}([C_{S}^{\bullet}(\mathcal{L}/K)]) = [C_{S}^{\bullet}(\mathcal{L}'/K)], \quad
\quot^{\mathcal{G}}_{\mathcal{G}/N}(\Phi_{S}(\mathcal{L}/K)) = \Phi_{S}(\mathcal{L}'/K).
\]
In particular, if the EIMC holds for $\mathcal{L}/K$, then it holds for $\mathcal{L}'/K$.
\item
Let $\mathcal{H}$ be an open subgroup of $\mathcal{G}$ and put $K' := \mathcal{L}^{\mathcal{H}}$. Then
\[
\res^{\mathcal{G}}_{\mathcal{H}}([C_{S}^{\bullet}(\mathcal{L}/K)]) = [C_{S}^{\bullet}(\mathcal{L}/K')], \quad
\res^{\mathcal{G}}_{\mathcal{H}}(\Phi_{S}(\mathcal{L}/K)) = \Phi_{S}(\mathcal{L}/K').
\]
In particular, if the EIMC holds for $\mathcal{L}/K$, then it holds for $\mathcal{L}/K'$.
\end{enumerate}
\end{prop}

\subsection{The EIMC over hybrid Iwasawa algebras}
We show how hybrid Iwasawa algebras can be used to `break up' certain cases of the EIMC.

\begin{theorem}\label{thm:EIMC-break-down}
Let $\mathcal{L}/K$ be an admissible one-dimensional $p$-adic Lie extension with Galois group $\mathcal{G}$.
Suppose that $\Lambda(\mathcal{G})$ is $N$-hybrid for some
finite normal subgroup $N$ of $\mathcal{G}$.
Let $\overline{\mathcal{P}}$ be a Sylow $p$-subgroup of $\overline{\mathcal{G}}:=\Gal(\mathcal{L}^{N}/K) \simeq \mathcal{G}/N$.
Then the following statements hold.
\begin{enumerate}
\item The EIMC holds for $\mathcal{L}/K$ if and only if it holds for $\mathcal{L}^{N}/K$.
\item The EIMC with uniqueness holds for $\mathcal{L}/K$ if and only if it holds for $\mathcal{L}^{N}/K$.
\item If $(\mathcal{L}^{N})^{\overline{\mathcal{P}}} / \Q$ is abelian, then the EIMC holds for $\mathcal{L}/K$.
\item If $\mathcal{L}^{N} / \Q$ is abelian, then the EIMC with uniqueness holds for $\mathcal{L}/K$.
\end{enumerate}
\end{theorem}

\begin{remark}
Under the hypotheses in (iii), $\overline{\mathcal{P}}$ is necessarily normal in $\overline{\mathcal{G}}$
and thus is its unique Sylow $p$-subgroup.
\end{remark}

\begin{proof}[Proof of Theorem \ref{thm:EIMC-break-down}]
By assumption $\Lambda(\mathcal{G})$ decomposes into $\Lambda(\mathcal{G}) e_{N} \oplus \mathfrak{M}(\mathcal{G}) (1 - e_{N})$
for some maximal order $\mathfrak{M}(\mathcal{G})$. This induces a decomposition of relative $K$-groups
\begin{eqnarray*}
K_{0}(\Lambda(\mathcal{G}), \mathcal{Q}(\mathcal{G}))
& \simeq & K_{0}(\Lambda(\mathcal{G}) e_{N}, \mathcal{Q}(\mathcal{G}) e_{N}) \times K_{0}(\mathfrak{M}(\mathcal{G}) (1 - e_{N}), \mathcal{Q}(\mathcal{G}) (1-e_{N}))\\
 & \simeq & K_{0}(\Lambda(\mathcal{G}/N), \mathcal{Q}(\mathcal{G}/N)) \times K_{0}(\mathfrak{M}(\mathcal{G}) (1 - e_{N}), \mathcal{Q}(\mathcal{G}) (1-e_{N}))
\end{eqnarray*}
which maps $[C_{S}^{\bullet}(\mathcal{L}/K)]$ to the pair
$([C_{S}^{\bullet}(\mathcal{L}^{N}/K)], [\mathfrak{M}(\mathcal{G}) (1 - e_{N}) \otimes_{\Lambda(\mathcal{G})}^{\mathbb{L}} C_{S}^{\bullet}(\mathcal{L}/K)])$
by Proposition \ref{prop:funtorialities} (i).
Similarly, we have a decomposition
\begin{eqnarray*}
\zeta(\mathcal{Q}(\mathcal{G}))^{\times} & \simeq & \zeta(\mathcal{Q}(\mathcal{G}) e_{N})^{\times} \times \zeta(\mathcal{Q}(\mathcal{G})(1-e_{N}))^{\times} \\
& \simeq & \zeta(\mathcal{Q}(\mathcal{G}/N))^{\times} \times \zeta(\mathcal{Q}(\mathcal{G})(1-e_{N}))^{\times}
\end{eqnarray*}
which maps $\Phi_{S}(\mathcal{L}/K)$ to the pair $(\Phi_{S}(\mathcal{L}^{N}/K), \Phi_{S}(\mathcal{L}/K) (1-e_{N}))$.
Hence (i) follows from Corollary \ref{cor:EIMC-MaxOrd} and Remark \ref{rmk:nr-surjective}.

Part (ii) follows from (i) if
\[
SK_{1}((1-e_{N})\mathcal{Q}(\mathcal{G})) := 
\ker(\nr: K_{1}((1-e_{N})\mathcal{Q}(\mathcal{G})) \longrightarrow \zeta((1-e_{N})\mathcal{Q}(\mathcal{G}))^{\times})
\]
vanishes.
This is indeed the case since Corollary \ref{cor:hybrid-implies-nr-surjective} implies that $\mathcal{Q}(\mathcal{G})(1-e_{N})$ is a direct product
of matrix rings over fields.
Part (iii) follows from part (i) and Corollary \ref{cor:EIMC-unconditional}.
Finally, part (iv) follows from parts (ii) and (iii) and Remark \ref{rmk:SK1}.
\end{proof}

The following theorem is useful in applications to proving results about finite Galois extensions of number fields.

\begin{theorem}\label{thm:EIMC-start-with-num-fields}
Let $L/K$ be a finite Galois extension of totally real number fields with Galois group $G$.
Let $p$ be an odd prime and let $L_{\infty}$ be the cyclotomic $\Z_{p}$-extension of $L$.
Let $P$ be a Sylow $p$-subgroup of $G$. Then the following statements hold.
\begin{enumerate}
\item $L_{\infty}/K$ is an admissible one-dimensional $p$-adic Lie extension.
\item If $p \nmid |G|$ then the EIMC with uniqueness holds for $L_{\infty}/K$.
\item If $L^{P}/\Q$ is abelian then the EIMC holds for $L_{\infty}/K$.
\end{enumerate}
Suppose further that $\Z_{p}[G]$ is $N$-hybrid.
Let $\overline{P}$ be a Sylow $p$-subgroup of $\overline{G}:=\Gal(L^{N}/K) \simeq G/N$.
Then the following statements hold.
\begin{enumerate}
\setcounter{enumi}{3}
\item If $L^{N}/\Q$ is abelian then the EIMC with uniqueness holds for $L_{\infty}/K$.
\item If $(L^{N})^{\overline{P}}/\Q$ is abelian then the EIMC holds for $L_{\infty}/K$.
\end{enumerate}
\end{theorem}

\begin{remark}
Under the hypothesis in (iii), $P$ is necessarily normal in $G$ and thus is its unique
Sylow $p$-subgroup. In part (v), we have $(L^{N})^{\overline{P}}=L^{NP}=L^{N \rtimes P}$.
\end{remark}

\begin{proof}[Proof of Theorem \ref{thm:EIMC-start-with-num-fields}]
For part (i) it is trivial to check that the conditions of Definition \ref{def:one-dim-adm} are satisfied.
We adopt the setup and notation of \S \ref{subsec:algebras-in-Iwasawa-thy}.
Since $H$ identifies with a subgroup  of $G$, part (ii) follows from Theorem \ref{thm:EIMC-for-p-not-dividing-ordH}
(in fact $H$ identifies with $G$ in this case.)
Let $\mathcal{P}$ be a Sylow $p$-subgroup of $\mathcal{G}$.
Since $\Gamma_{L}:=\Gal(L_{\infty}/L)$ is an open normal pro-$p$ subgroup of $\mathcal{G}$
we have $\Gamma_{L} \leq \mathcal{P}$ and $\mathcal{P}$ maps to $P$ under the natural 
projection $\mathcal{G} \rightarrow \mathcal{G}/\Gamma_{L} \simeq G$.
Hence $L^{P}=L_{\infty}^{\mathcal{P}}$ and so part (iii) follows from Corollary \ref{cor:EIMC-unconditional}.

Now suppose that $\Z_{p}[G]$ is $N$-hybrid.
Proposition \ref{prop:hybrid-codescent} says that $N$ identifies with a normal subgroup of $\mathcal{G}$ which is also a normal subgroup of $H$;
moreover, both $\Z_{p}[H]$ and $\Lambda(\mathcal{G})$ are $N$-hybrid. 
Note that $(L_{\infty})^{N}=(L^{N})_{\infty}$.
If $L^{N}/\Q$ is abelian then $L_{\infty}^{N}/\Q$ is also abelian and so part (iv) follows from 
Theorem \ref{thm:EIMC-break-down} (iv).
Part (v) now follows from (iii) and Theorem \ref{thm:EIMC-break-down} (i).
\end{proof}

\begin{corollary}\label{cor:EIMC-Frobenius}
Let $L/K$ be a finite Galois extension of totally real number fields with Galois group $G$.
Suppose that $G = N \rtimes V$ is a Frobenius group with Frobenius kernel $N$ and abelian Frobenius complement $V$.
Further suppose that $L^{N}/\Q$ is abelian (in particular, this is the case when $K=\Q$).
Let $p$ be an odd prime and let $L_{\infty}$ be the cyclotomic $\Z_{p}$-extension of $L$.
Then the following statements hold.
\begin{enumerate}
\item If $p \nmid |N|$ then the EIMC with uniqueness holds for $L_{\infty}/K$.
\item If $N$ is a $p$-group then the EIMC holds for $L_{\infty}/K$.
\item If $N$ is an $\ell$-group for any prime $\ell$ then the EIMC holds for $L_{\infty}/K$.\\
(This includes the cases $\ell=2$ and $\ell=p$.)
\end{enumerate}
In particular, (iii) holds in the following cases:
\begin{itemize}
\item $G \simeq \Aff(q)$, where $q$ is a prime power (see Example \ref{ex:affine}),
\item $G \simeq C_{\ell} \rtimes C_{q}$, where $q<\ell$ are distinct primes such that $q \mid (\ell-1)$ and $C_{q}$ acts on $C_{\ell}$ via an embedding $C_{q} \hookrightarrow \Aut(C_{\ell})$ (see Example \ref{ex:metacyclic}),
\item $G$ is isomorphic to any of the Frobenius groups constructed in Example \ref{ex:non-abelian-kernel}.
\end{itemize}
\end{corollary}

\begin{proof}
Part (i) follows from Proposition \ref{prop:frob-N-hybrid} and Theorem \ref{thm:EIMC-start-with-num-fields} (iv).
Part (ii) follows from  Theorem \ref{thm:EIMC-start-with-num-fields} (iii) with $P=N$.
Part (iii) is just the combination of (i) and (ii).
\end{proof}

\begin{remark}
Let $L$ be a totally real finite Galois extension of $\Q$ with Galois group $G$.
Let $p$ be an odd prime and let $L_{\infty}$ be the cyclotomic $\Z_{p}$-extension of $L$.
If $G$ is abelian then the EIMC with uniqueness for $L_{\infty}/\Q$ holds by Corollary \ref{cor:EIMC-unconditional} and Remark \ref{rmk:SK1}.
If $G$ is an $\ell$-group for any prime $\ell$ then EIMC for $L_{\infty}/\Q$ holds: if $p=\ell$ this follows from the fact that the $\mu=0$
hypothesis is known in this case (see Remark \ref{rmk:mu=0}), and if $p \neq \ell$ then the EIMC with uniqueness for $L_{\infty}/\Q$
holds by Theorem \ref{thm:EIMC-for-p-not-dividing-ordH}. 
Thus it remains to consider finite groups $G$ that are neither abelian nor $\ell$-groups for any prime $\ell$. 
Among all such groups of order $\leq 10^{6}$ there are $568,220$ metabelian Frobenius groups
(see \cite[Remark 11.13 (A)]{MR1600514}), which in particular have abelian Frobenius complement
and so satisfy the hypotheses of Corollary \ref{cor:EIMC-Frobenius}. 
In fact, there are many more pairs ($G$, $p$) with $p$ an odd prime dividing $|G| \leq 10^{6}$ that satisfy the hypotheses
of Theorem \ref{thm:EIMC-start-with-num-fields} (v) in the case $K=\Q$ (in particular, recall that if $\Z_{p}[G]$ is $N$-hybrid for some non-trivial $N$ then $G$ need not be a Frobenius group).
\end{remark}

\begin{example} \label{ex:EIMC-dicyclic}
Let $p$ be an odd prime.
Let $V = \Dic_{p}$ be dicyclic of order $4p$ and $G = N \rtimes V$ be a Frobenius group as in Example \ref{ex:dicyclic-complement}.
Let $L/K$ be a finite Galois extension of totally real number fields with Galois group $\Gal(L/K) \simeq G$.
The unique Sylow $p$-subgroup $P$ of $\Dic_{p}$ coincides with the commutator subgroup
and thus $L^{N \rtimes P}/K$ is cyclic of order $4$.
If we further assume that $L^{N \rtimes P}/\Q$ is abelian (for instance, assume that $K=\Q$),
then Theorem \ref{thm:EIMC-start-with-num-fields} (v) implies that the EIMC holds for $L_{\infty}/K$,
where $L_{\infty}$ is the cyclotomic $\Z_{p}$-extension of $L$.
Note that this cannot be deduced from Corollary \ref{cor:EIMC-Frobenius}
because although $G$ is a Frobenius group, it has a non-abelian Frobenius complement.
\end{example}

\begin{example} \label{ex:EIMC-modified-affine}
Let $q=\ell^{n}$ be a prime power and consider the group $\F_{q} \rtimes (\F_{q}^{\times} \rtimes \langle \phi \rangle)$
of Example \ref{ex:affine-Frobenius}. Let $p \mid (q-1)$ be an odd prime that does not divide $n$.
Then the group $\mu_{p}(\F_{q})$ of $p$-power roots of unity in $\F_{q}$ is non-trivial and we put
$G := \F_{q} \rtimes (\mu_{p}(\F_{q}) \rtimes \langle \phi \rangle)$ and $U := \F_{q} \rtimes \mu_{p}(\F_{q}) \unlhd G$.
Then Example \ref{ex:affine-Frobenius} and Proposition \ref{prop:hybrid-basechange-up} imply that $\Z_{p}[G]$ is
$\F_{q}$-hybrid.
Now assume that $L/K$ is a Galois extension of totally real number fields with $\Gal(L/K) \simeq G$
and let $L_{\infty}$ be the cyclotomic $\Z_{p}$-extension of $L$.
If $L^{U}/\Q$ is abelian (for instance if $K=\Q$), then Theorem \ref{thm:EIMC-start-with-num-fields} (v) implies that the EIMC holds
for $L_{\infty}/K$.
\end{example}

\begin{example}\label{ex:EIMC-S4}
Let $L/K$ be a finite Galois extension of totally real number fields with Galois group $\Gal(L/K) \simeq S_{4}$.
Let $p$ be an odd prime and let $L_{\infty}$ be the cyclotomic $\Z_{p}$-extension of $L$.
Then Theorem \ref{thm:EIMC-start-with-num-fields} (ii) shows that  the EIMC with uniqueness holds for $L_{\infty}/K$ when $p>3$.
Now further assume that $L^{A_{4}}/\Q$ is abelian (in particular, this is the case when $K=\Q$) and consider the case $p=3$.
The group ring $\Z_{3}[S_{4}]$ is $V_{4}$-hybrid as shown in Example \ref{ex:S4-A4-V4}.
Moreover, the Sylow $3$-subgroup of $S_{4}/V_{4} \simeq S_{3}$ is $A_{3} \simeq C_{3}$ and we have
$(L^{V_{4}})^{A_{3}}=L^{A_{4}}$, so the EIMC for $L_{\infty}/K$ follows from Theorem \ref{thm:EIMC-start-with-num-fields} (v).
Let $\mathcal{G}=\Gal(L_{\infty}/K)$.
Then as shown in Example \ref{ex:S4-V4-II} we have $\mathcal{G} \simeq S_{4} \times \Gamma_{K}$,
and so no skewfields occur in the Wedderburn decomposition of $\mathcal{Q}(\mathcal{G})$.
Therefore, the EIMC with uniqueness also holds for $L_{\infty}/K$ when $p=3$.
\end{example}

\subsection{Remarks on the higher dimension case} \label{subsec:higher-rk}
In \cite{MR3091976}, Kakde proved a more general version of the EIMC for admissible
$p$-adic Lie extensions of arbitrary (finite) dimension under a suitable version of the $\mu=0$ hypothesis.
This used a strategy of Burns and Kato to reduce the proof to the one-dimensional case discussed above (see \cite{MR3294653}).
We briefly discuss some of the obstacles to generalising the approach of this article to prove higher dimension cases
of the EIMC when a suitable $\mu=0$ hypothesis is not known.
A serious obstacle is that a certain `$\mathfrak{M}_{H}(G)$-conjecture' is required to even formulate
the higher dimension version of the EIMC, and that this is presently only known to hold under a suitable $\mu=0$ hypothesis
(see \cite[p.\ 5]{zbMATH06148870} and \cite{MR2905532}).
Another problem is that a higher dimension version of Theorem \ref{thm:EIMC-MaxOrd} (the `maximal order variant of the EIMC')
has not been proven unconditionally.
Finally, and perhaps most importantly, it is not clear how the notion of a hybrid Iwasawa algebra generalises to the higher dimension case.

\bibliography{hybrid-iwasawa-Bib}{}
\bibliographystyle{amsalpha}

\end{document}